\documentclass[11pt,letterpaper]{amsart}

\makeatletter
\def\subsection{\@startsection{subsection}{2}%
  \z@{.7\linespacing\@plus.7\linespacing}{.2\linespacing}%
  {\centering\normalfont\scshape}}
\makeatother

\usepackage[letterpaper,top=3.7cm, bottom=3.7cm, left=3.7cm, right=3.7cm]{geometry}

\usepackage{amssymb}
\usepackage{mathtools}
\usepackage{mathabx}
\usepackage{enumerate}
\usepackage{bbm}

\newtheorem{theorem}{Theorem}[section]
\newtheorem{proposition}[theorem]{Proposition}

\newtheorem{lemma}[theorem]{Lemma}
\newtheorem{conjecture}[theorem]{Conjecture}
\theoremstyle{definition}

\newtheorem*{remark*}{Remark}

\numberwithin{equation}{section}

\renewcommand{\AA}{\mathbb{A}}
\newcommand{\ZZ}{\mathbb{Z}}
\newcommand{\FF}{\mathbb{F}}
\newcommand{\CC}{\mathbb{C}}
\newcommand{\RR}{\mathbb{R}}

\newcommand{\Ac}{\mathcal{A}}
\newcommand{\1}{\mathbbm{1}}

\newcommand{\ulambda}{\underline{\lambda}}
\newcommand{\umu}{\underline{\mu}}
\newcommand{\unu}{\underline{\nu}}
\newcommand{\utau}{\underline{\tau}}
\newcommand{\ukappa}{\underline{\kappa}}
\newcommand{\usigma}{\underline{\sigma}}
\renewcommand{\utau}{\underline{\tau}}
\newcommand{\urho}{\underline{\rho}}

\newcommand{\sums}[1]{\sum_{\substack{#1}}}
\newcommand{\bigcups}[1]{\bigcup_{\substack{#1}}}
\newcommand{\ang}[1]{\langle#1\rangle}
\newcommand{\norm}[1]{\lVert#1\rVert}

\newcommand{\abs}[1]{\lvert#1\rvert}

\newcommand{\biggabs}[1]{\bigg\lvert#1\bigg\rvert}
\newcommand{\Biggabs}[1]{\Bigg\lvert#1\Bigg\rvert}

\DeclareMathOperator{\Ind}{Ind}
\DeclareMathOperator{\rk}{rk}
\DeclareMathOperator{\GL}{GL}
\DeclareMathOperator{\PGL}{PGL}
\DeclareMathOperator{\Par}{Par}
\DeclareMathOperator{\re}{Re}

\newcommand{\qbin}[2]{\genfrac{[}{]}{0pt}{}{{#1}}{{#2}}_q}

\begin{document}

\title[Intersection theorems for finite general linear groups]{Intersection theorems for finite\\general linear groups}

\author{Alena Ernst and Kai-Uwe Schmidt}
\address{Department of Mathematics, Paderborn University, Warburger Str.\ 100, 33098 Paderborn, Germany.}
\email{alena.ernst@math.upb.de}
\email{kus@math.upb.de}

\thanks{Funded by the Deutsche Forschungsgemeinschaft (DFG, German Research Foundation) -- Project number 459964179.}

\date{17 May 2022, updated 31 August 2022 and 02 February 2023}


\begin{abstract}
A subset $Y$ of the general linear group $\GL(n,q)$ is called $t$-intersecting if $\rk(x-y)\le n-t$ for all $x,y\in Y$, or equivalently $x$ and $y$ agree pointwise on a $t$-dimensional subspace of $\FF_q^n$ for all $x,y\in Y$. We show that, if $n$ is sufficiently large compared to $t$, the size of every such $t$-intersecting set is at most that of the stabiliser of a basis of a $t$-dimensional subspace of $\FF_q^n$. In case of equality, the characteristic vector of $Y$ is a linear combination of the characteristic vectors of the cosets of these stabilisers. We also give similar results for subsets of $\GL(n,q)$ that intersect not necessarily pointwise in $t$-dimensional subspaces of $\FF_q^n$ and for cross-intersecting subsets of $\GL(n,q)$. These results may be viewed as variants of the classical Erd\H{o}s-Ko-Rado Theorem in extremal set theory and are $q$-analogs of corresponding results known for the symmetric group. Our methods are based on eigenvalue techniques to estimate the size of the largest independent sets in graphs and crucially involve the representation theory of $\GL(n,q)$.
\end{abstract}

\maketitle

\thispagestyle{empty}


\section{Introduction and results}

One of the most famous results in extremal set theory is the Erd\H{o}s-Ko-Rado Theorem~\cite{ErdKoRad1961}. In its strengthened version~\cite{Wil1984} it states that, for all fixed $k$ and $t$ and all sufficiently large $n$, every $t$-intersecting family of $k$-subsets of $\{1,2,\dots,n\}$ has size at most $\binom{n-t}{k-t}$ and equality holds if and only if there are $t$ distinct points of $\{1,2,\dots,n\}$ contained in all members of the family.
\par
There are several analogs of the Erd\H{o}s-Ko-Rado Theorem (see~\cite{GodMea2016}, for example). Most notably, following important earlier work~\cite{DezFra1977},~\cite{CamKu2003},~\cite{LarMal2004},~\cite{GodMar2009}, a corresponding result for the symmetric group $S_n$ was obtained by Ellis, Friedgut, and Pilpel in a landmark paper~\cite{EllFriPil2011}. A subset $Y$ of $S_n$ is \emph{$t$-intersecting} if, for all $x,y\in Y$, there exist distinct $i_1,i_2,\dots,i_t$ in $\{1,2,\dots,n\}$ such that $x(i_k)=y(i_k)$ for all $k$. It was shown in~\cite{EllFriPil2011} that, for each fixed $t$ and all sufficiently large $n$, every $t$-intersecting set in $S_n$ has size at most $(n-t)!$ and equality holds if and only if $Y$ is a coset of the stabiliser of a $t$-tuple of distinct points in $\{1,2,\dots,n\}$. 
\par
In this paper we consider a $q$-analog of this problem, namely we study a corresponding problem for the finite general linear groups. Throughout this paper $q$ is a fixed prime power and $G_n$ denotes the general linear group of degree $n$ over the finite field~$\FF_q$, namely the group of invertible $n\times n$ matrices over $\FF_q$. We say that two elements $x,y\in G_n$ are \emph{$t$-intersecting} if there exist linearly independent elements $u_1,u_2,\dots,u_t$ in $\FF_q^n$ such that $xu_k=yu_k$ for all $k$. Equivalently $x,y\in G_n$ are $t$-intersecting if $\rk(x-y)\le n-t$. A subset~$Y$ of $G_n$ is called \emph{$t$-intersecting} if all pairs in $Y\times Y$ are $t$-intersecting. 
\par
A coset of the stabiliser of a $t$-tuple of linearly independent elements of $\FF_q^n$ has the form
\[
\{g\in G_n:gu_1=v_1,\dots,gu_t=v_t\}
\]
for some $t$-tuples $(u_1,u_2,\dots,u_t)$ and $(v_1,v_2,\dots,v_t)$ of linearly independent elements of $\FF_q^n$. We call such a coset a \emph{$t$-coset}. It is plain that every $t$-coset is $t$-intersecting. Note that the size of a $t$-coset is
\begin{equation}
\prod_{i=t}^{n-1}(q^n-q^i).   \label{eqn:extremal_size}
\end{equation}
The $t$-cosets are however not the only $t$-intersecting sets of this size in $G_n$, as the transpose of every $t$-intersecting set is $t$-intersecting.
\par
We shall often identify a subset~$Y$ of $G_n$ with its characteristic vector $1_Y\in\CC(G_n)$ (where $\CC(G_n)$ is the vector space of functions from $G_n$ to $\CC$). It is well known (see~\cite{AhaAha2014} or~\cite{AhmMea2015}, for example) that, since $G_n$ contains a Singer cycle as a regular subgroup, the size of every $1$-intersecting set in $G_n$ is at most the expression given in~\eqref{eqn:extremal_size} for $t=1$. Meagher and Razafimahatratra~\cite{MeaRaz2021} have shown that, if $Y$ is a $1$-intersecting set of size $q^2-q$ in $G_2$, then $1_Y$ is in the span of the characteristic vectors of the $1$-cosets. We prove a corresponding result for all $t$ and $n$ for which $n$ is sufficiently large compared to $t$.
\begin{theorem}
\label{thm:main_intersection}
Let $t$ be a positive integer and let $Y$ be a $t$-intersecting set in $ G_n$. If~$n$ is sufficiently large compared to $t$, then
\[
\abs{Y}\le \prod_{i=t}^{n-1}(q^n-q^i)
\]
and, in case of equality, $1_Y$ is spanned by the characteristic vectors of $t$-cosets.
\end{theorem}
\par
We also prove a result on cross-intersecting subsets of $G_n$. Two subsets $Y$ and $Z$ are \emph{$t$-cross-intersecting} if all pairs in $Y\times Z$ are $t$-intersecting.
\begin{theorem}
\label{thm:main_cross_intersection}
Let $t$ be a positive integer and let $Y$ and $Z$ be $t$-cross-intersecting sets in~$ G_n$. If~$n$ is sufficiently large compared to $t$, then
\[
\sqrt{\abs{Y}\cdot\abs{Z}}\le \prod_{i=t}^{n-1}(q^n-q^i)
\]
and, in case of equality, $1_Y$ and $1_Z$ are spanned by the characteristic vectors of $t$-cosets.
\end{theorem}
\par
Theorems~\ref{thm:main_intersection} and~\ref{thm:main_cross_intersection} may be seen as $q$-analogs of~\cite[Thm.~5 and~6]{EllFriPil2011}. It seems plausible that corresponding $q$-analogs of~\cite[Thm.~3 and~4]{EllFriPil2011} also hold. In the case of $t$-intersecting sets, this means that the extremal $t$-intersecting sets in $G_n$ are the $t$-cosets and their transposes whenever~$n$ is sufficiently large compared to $t$. In fact, Ahanjideh~\cite{Aha2022} has shown that every $1$-intersecting set in $G_2$ of size $q^2-q$ must be either a $1$-coset or the transpose of a $1$-coset. We therefore pose the following conjectures.
\begin{conjecture}
\label{con:t-intersecting}
Let $Y$ be a $t$-intersecting set in $G_n$ whose size meets the bound in Theorem~\ref{thm:main_intersection}. If $n$ is sufficiently large compared to $t$, then $Y$ or $Y^T$ is a $t$-coset.
\end{conjecture}
\par
\begin{conjecture}
Let $Y$and $Z$ be $t$-cross-intersecting sets in $G_n$ whose sizes meet the bound in Theorem~\ref{thm:main_cross_intersection}. If $n$ is sufficiently large compared to $t$, then $Y=Z$ and $Y$ or~$Y^T$ is a $t$-coset.
\end{conjecture}
\par
A subset $Y$ of the symmetric group $S_n$ is \emph{$t$-set-intersecting} if, for all $x,y\in Y$, there is a subset $I$ of $\{1,2,\dots,n\}$ containing $t$ elements such that $x(I)=y(I)$. It was shown in~\cite{Ell2012} that, for each fixed $t$ and all sufficiently large $n$, every $t$-set-intersecting set in $S_n$ has size at most $t!(n-t)!$ and equality holds if and only if $Y$ is a coset of the stabiliser of a subset of $\{1,2,\dots,n\}$ containing $t$ elements. 
\par
We also obtain a $q$-analog of this result. We say that two elements $x,y\in G_n$ are \emph{$t$-space-intersecting} if there exists a $t$-dimensional subspace $U$ of $\FF_q^n$ (or \emph{$t$-space} for short) such that $xU=yU$. A subset~$Y$ of $G_n$ is called \emph{$t$-space-intersecting} if all pairs in $Y\times Y$ are $t$-space-intersecting. Of course in this context it would be more natural to replace~$G_n$ by the projective linear group $\PGL(n,q)$. However results for~$G_n$ and for $\PGL(n,q)$ can be easily translated into each other and for consistency we prefer to work with~$G_n$. A coset of the stabiliser in $G_n$ of a $t$-space is clearly $t$-space-intersecting and has order
\begin{equation}
\Bigg[\prod_{i=0}^{t-1}(q^t-q^i)\Bigg]\Bigg[\prod_{i=t}^{n-1}(q^n-q^i)\Bigg].   \label{eqn:size_stab_space}
\end{equation}
Note that again the transpose of a $t$-space-intersecting set is $t$-space-intersecting. The transpose of the stabiliser of a $t$-space is in fact the stabiliser of an $(n-t)$-space, so the stabiliser of an $(n-t)$-space is an example of a $t$-space-intersecting set that has the same size as that of the stabiliser of a $t$-space.
\par
Using an argument involving a Singer cycle, similarly as that above, Meagher and Spiga~\cite{MeaSpi2011} have shown that the size of every $1$-space-intersecting set in $G_n$ is at most the expression given in~\eqref{eqn:size_stab_space} for $t=1$. We show that this is true for all $t$ and all sufficiently large $n$.
\begin{theorem}
\label{thm:main_intersection_space}
Let $t$ be a positive integer and let $Y$ be a $t$-space-intersecting set in~$G_n$. If~$n$ is sufficiently large compared to $t$, then
\[
\abs{Y}\le \Bigg[\prod_{i=0}^{t-1}(q^t-q^i)\Bigg]\Bigg[\prod_{i=t}^{n-1}(q^n-q^i)\Bigg]
\]
and, in case of equality, $1_Y$ is spanned by the characteristic vectors of cosets of stabilisers of $t$-spaces.
\end{theorem}
\par
Again, we have a corresponding result on cross-intersecting subsets of $G_n$, in which we call two subsets $Y$ and $Z$ of $G_n$ \emph{$t$-space-cross-intersecting} if all pairs in $Y\times Z$ are $t$-space-intersecting. 
\begin{theorem}
\label{thm:main_cross_intersection_space}
Let $t$ be a positive integer and let $Y$ and $Z$ be $t$-space-cross-intersecting sets in~$ G_n$. If~$n$ is sufficiently large compared to $t$, then
\[
\sqrt{\abs{Y}\cdot\abs{Z}}\le \Bigg[\prod_{i=0}^{t-1}(q^t-q^i)\Bigg]\Bigg[\prod_{i=t}^{n-1}(q^n-q^i)\Bigg]
\]
and, in case of equality, $1_Y$ and $1_Z$ are spanned by the characteristic vectors of cosets of stabilisers of $t$-spaces.
\end{theorem}
\par
Meagher and Spiga~\cite{MeaSpi2011} conjectured that the extremal $1$-space-intersecting sets in~$G_n$ must be cosets of the stabiliser of a $1$-space or cosets of the stabiliser of an $(n-1)$-space. This was proved by the same authors for $n=2$ ~\cite{MeaSpi2011} and $n=3$~\cite{MeaSpi2014} and by Spiga for all $n\ge 4$~\cite{Spi2019}. We therefore pose the following conjectures.
\begin{conjecture}
Let $Y$ be a $t$-space-intersecting set in $G_n$ whose size meets the bound in Theorem~\ref{thm:main_intersection_space}. If $n$ is sufficiently large compared to $t$, then $Y$ is a coset of the stabiliser of a $t$-space or a coset of the stabiliser of an $(n-t)$-space.
\end{conjecture}
\par
\begin{conjecture}
Let $Y$and $Z$ be $t$-space-cross-intersecting sets in $G_n$ whose sizes meet the bound in Theorem~\ref{thm:main_cross_intersection_space}. If $n$ is sufficiently large compared to $t$, then $Y=Z$ and $Y$ is the stabiliser of a $t$-space or the stabiliser of an $(n-t)$-space.
\end{conjecture}
\par
Not surprisingly, as in~\cite{EllFriPil2011} and~\cite{Ell2012}, our proofs are based on eigenvalue techniques, in particular weighted versions of the Hoffman bound on independent sets in graphs, and crucially involve the representation theory of $G_n$. We organise this paper as follows. In Section~\ref{sec:GL} we summarise relevant background on the representation theory of $G_n$. In Section~\ref{sec:association_schemes} we recall versions of the Hoffman bound from~\cite{EllFriPil2011} and explain how they can be applied in our setting. In Section~\ref{sec:matrix} we prepare some key steps of the proofs of our main results and in particular study properties of a matrix related to the character table of $G_n$. Sections~\ref{sec:proofs_points} and~\ref{sec:proofs_spaces} contain the main arguments of our proofs 
of Theorems~\ref{thm:main_intersection} and~\ref{thm:main_cross_intersection} and Theorems~\ref{thm:main_intersection_space} and~\ref{thm:main_cross_intersection_space}, respectively. In Section~\ref{sec:estimates} we prove some auxiliary ingredients used in our proofs.
\par
We close this introduction by noting that, after a first version of this paper was made publically available, Ellis, Kindler, and Lifshitz~\cite{EllKinLif2022} independently proved a result that is slightly more general than Theorem~\ref{thm:main_intersection} and also proved Conjecture~\ref{con:t-intersecting}. Their methods are completely different compared to ours and in particular make no use of the representation theory of $G_n$.


\section{The finite general linear groups}
\label{sec:GL}

In this section we mostly recall some relevant facts about the conjugacy classes and the character theory of $ G_n$.

\subsection{Partitions}

An (integer) \emph{partition} is a sequence $\lambda=(\lambda_1,\lambda_2,\dots)$ of nonnegative integers satisfying $\lambda_1\ge\lambda_2\ge\cdots$. The set of partitions is denoted by $\Par$. We often omit trailing zeros and write $\lambda=(\lambda_1,\lambda_2,\dots,\lambda_k)$ if $\lambda_k>0$ and $\lambda_{k+1}=0$. The \emph{size} of $(\lambda_1,\lambda_2,\dots)$ is defined to be $\abs{\lambda}=\lambda_1+\lambda_2+\cdots$. If $\abs{\lambda}=n$, then we also say that $\lambda$ is a partition of $n$. We denote the unique partition of $0$ by $\varnothing$. 
\par
The \emph{Young diagram} of a partition $(\lambda_1,\lambda_2,\dots,\lambda_k)$ of $n$ is an array of~$n$ boxes with left-justified rows and top-justified columns, where row $i$ contains $\lambda_i$ boxes. To each partition $\lambda$ belongs a \emph{conjugate partition} $\lambda'$ whose parts are the number of boxes in the columns of the Young diagram of $\lambda$. For two partitions $\lambda=(\lambda_1,\lambda_2,\dots)$ and $\mu=(\mu_1,\mu_2,\dots)$ of the same size, we say that $\lambda$ \emph{dominates} $\mu$ and write $\lambda\unrhd\mu$ if
\[
\sum_{i=1}^k\lambda_i\ge\sum_{i=1}^k\mu_i\quad\text{for each $k\ge 1$}.
\]
This indeed defines a partial order on the set of partitions of a fixed size, which is called the \emph{dominance order}.


\subsection{Conjugacy classes}

We shall now describe the conjugacy classes of $ G_n$ (see~\cite[Ch.~IV,\S~3]{Mac1979}, for example). Let $\Phi$ be the set of monic irreducible polynomials in $\FF_q[X]$ distinct from $X$. For $a\in\FF_q^*$ (where $\FF_q^*$ is the multiplicative group of $\FF_q$), we shall often write $a$ instead of $X-a$ when the meaning is clear from the context. We also write $\abs{f}$ for the degree of $f\in\Phi$. Let $\Lambda$ be the set of mappings $\ulambda\colon\Phi\to\Par$ of finite support (with $\varnothing$ being the zero element in $\Par$). We define the \emph{size} of such a mapping to be
\[
\norm{\ulambda}=\sum_{f\in\Phi}\abs{\ulambda(f)}\cdot\abs{f}
\]
and put $\Lambda_n=\{\ulambda\in\Lambda:\norm{\ulambda}=n\}$. The \emph{companion matrix} of $f\in\Phi$ with $f=X^d+f_{d-1}X^{d-1}+\cdots+f_1X+f_0$ is
\[
C(f)=\begin{bmatrix}
  &   &  &  & -f_0\\
1 &  & & & -f_1\\
  & 1 & & & -f_2\\
  &    & \ddots & & \vdots\\
  & & & 1 & -f_{d-1}\\
\end{bmatrix}\in\FF_q^{d\times d}
\]
(where blanks are filled with zeros). For $f\in\Phi$ of degree $d$ and a positive integer $k$, we write
\[
C(f,k)=\begin{bmatrix}
C(f) & I &   & & \\
  & C(f) & I & & \\
&&\ddots&\ddots\\
  &   &   & \ddots &  I\\
  &   &   & &  C(f) 
\end{bmatrix}\in\FF_q^{kd\times kd},
\]
where $I$ is an identity matrix of the appropriate size. For $f\in\Phi$ and $\sigma\in\Par$, we define $C(f,\sigma)$ to be the block diagonal matrix of order $\abs{\sigma}\cdot\abs{f}$ with blocks $C(f,\sigma_1),C(f,\sigma_2),\dots$. Finally, with every $\usigma\in\Lambda_n$ we associate the block diagonal matrix $R_{\usigma}$ of order $n$ whose blocks are $C(f,\usigma(f))$, where~$f$ ranges through the support of $\usigma$. Then every element $g$ of $ G_n$ is conjugate to exactly one matrix $R_{\usigma}$ for $\usigma\in\Lambda_n$, which is called the \emph{Jordan canonical form} of $g$. Hence~$\Lambda_n$ indexes the conjugacy classes of $ G_n$; we denote by $C_{\usigma}$ the conjugacy class containing~$R_{\usigma}$. The following result gives an explicit expression for the number of elements in $C_{\usigma}$.
\begin{lemma}[{\cite[Thm.~1.10.7]{Sta2012}}]
\label{lem:sizes_cc}
For each $\usigma\in\Lambda_n$, we have
\[
\frac{\abs{ G_n}}{\abs{C_{\usigma}}}=\prod_{f\in\Phi} \prod_{i=1}^{\abs{\usigma(f)}} \prod_{j=1}^{m_i(\usigma(f))} q^{\abs{f}\, s_i(\usigma(f)')} (1-q^{-\abs{f}\,j}),
\]
where $m_i(\sigma)= \abs{\{j\ge 1\colon\sigma_j=i\}}$ and $s_i(\sigma)=\sum_{j=1}^i \sigma_j$ for a partition $\sigma$.
\end{lemma}


\subsection{Parabolic induction}

Recall that, given a finite group $G$, a subgroup $H$ of $G$, and a class function $\phi$ on~$H$, the \emph{induced} class function $\Ind_H^G(\phi)$ on $G$ is given by
\begin{equation}
\Ind_H^G(\phi)(g)=\frac{1}{\abs{H}}\sums{x\in G\\xgx^{-1}\in H}\phi(xgx^{-1}).   \label{eqn:induction_formula}
\end{equation}
The character theory of $ G_n$ crucially relies on the induction of characters from parabolic subgroups of $ G_n$. 
\par
A \emph{composition} is much like a partition, except that the parts do not need to be nonincreasing. Let $\lambda=(\lambda_1,\lambda_2,\dots,\lambda_k)$ be a composition of~$n$. Let $P_\lambda$ be the parabolic subgroup of $ G_n$ consisting of block upper-triangular matrices with block sizes $\lambda_1,\lambda_2,\dots,\lambda_k$, namely
\begin{equation}
\label{eqn:def_parabolic_subgroup}
P_\lambda=\left\{\begin{bmatrix}
g_1 & *     & \cdots & *\\
       & g_2 & \cdots & *\\
       &        & \ddots & \vdots\\
       &        &            & g_k\\
\end{bmatrix}
:g_i\in G_{\lambda_i}
\right\}.
\end{equation}
Let $\pi_i:P_\lambda\to G_{\lambda_i}$ be the mapping that projects to the $i$-th diagonal block, so that
\begin{equation}
\pi_i:
\begin{bmatrix}
g_1 & *     & \cdots & *\\
       & g_2 & \cdots & *\\
       &        & \ddots & \vdots\\
       &        &            & g_k\\
\end{bmatrix}
\mapsto g_i.   \label{eqn:projections}
\end{equation}
Let $\phi_i$ be a class function on $G_{\lambda_i}$. Then 
\[
\prod_{i=1}^k(\phi_i\circ \pi_i)
\]
is a class function on $P_\lambda$. We define the product $\phi_1\odot\phi_2\odot\cdots\odot\phi_k$ to be the induction of this class function to $ G_n$, that is
\begin{equation}
\bigodot_{i=1}^k\phi_i=\Ind_{P_\lambda}^{ G_n}\left(\prod_{i=1}^k(\phi_i\circ \pi_i)\right).   \label{eqn:def_parabolic_induction}
\end{equation}


\subsection{Character theory of $ G_n$}
\label{sec:character_theory}

The complete set of complex irreducible characters has been obtained by Green~\cite{Gre1956}. A good treatment of this topic is also contained in~\cite[Ch.~IV]{Mac1979}. The complex irreducible representations were obtained by Gelfand~\cite{Gel1970} and the irreducible representations over fields of nondefining characteristic were obtained by James~\cite{Jam1986}. The approach of~\cite{Jam1986} is in fact very similar to the standard combinatorial approach to obtain the complex irreducible representations of the symmetric group (see~\cite{Sag2001}, for example) and we mostly follow~\cite{Jam1986} to recall some relevant background on the complex characters of $ G_n$.
\par
The irreducible characters of $ G_n$ are naturally indexed by $\Lambda_n$ and, for $\ulambda\in\Lambda_n$, we denote by $\chi^{\ulambda}$ the corresponding irreducible character. We shall use the short-hand notation $\chi^{f\mapsto\lambda}$ for $\chi^{\ulambda}$ if $\ulambda$ is supported only on $f\in\Phi$ and $\ulambda(f)=\lambda$. These are typically called the \emph{primary} irreducible characters of $ G_n$. It is well known (see~\cite[\S~8]{Jam1986}, for example) that the irreducible characters of $ G_n$ satisfy
\begin{equation}
\chi^{\ulambda}=\bigodot_{f\in\Phi}\chi^{f\mapsto\ulambda(f)}.   \label{eqn:irr_char_induction}
\end{equation}
In order to construct the primary irreducible characters, James~\cite{Jam1986} constructs characters of $G_{dm}$, denoted by $\xi^{f\mapsto\mu}$, where $f\in\Phi$ has degree~$d$ and~$\mu$ is a partition of $m$. Writing $\mu=(\mu_1,\mu_2,\dots,\mu_k)$, these characters satisfy~\cite[(6.2)]{Jam1986}
\begin{equation}
\xi^{f\mapsto\mu}=\bigodot_{i=1}^k\xi^{f\mapsto(\mu_i)}   \label{eqn:xi_induction_from_parts}
\end{equation}
and~\cite[(7.19)]{Jam1986}
\begin{equation}
\xi^{f\mapsto\mu}=\sum_{\lambda}K_{\lambda\mu}\,\chi^{f\mapsto\lambda},   \label{eqn:xi_chi}
\end{equation}
where $\lambda$ ranges over the partitions of $\abs{\mu}$ and $K_{\lambda\mu}$ is a \emph{Kostka number}, which equals the number of semistandard Young tableaux of shape $\lambda$ and content $\mu$. It is well known (see~\cite[\S~2.11]{Sag2001}, for example) that the Kostka numbers satisfy
\begin{equation}
\text{$K_{\mu\mu}=1$ and $K_{\lambda\mu}\ne 0\Rightarrow \lambda\unrhd\mu$}.   \label{eqn:Kostka_triangular}
\end{equation}
Conversely it is readily verified that there are integers $H_{\mu\lambda}$ satisfying
\begin{equation}
\chi^{f\mapsto\lambda}=\sum_{\mu}H_{\mu\lambda}\,\xi^{f\mapsto\mu}   \label{eqn:chi_xi}
\end{equation}
and
\begin{equation}
\text{$H_{\lambda\lambda}=1$ and $H_{\mu\lambda}\ne 0\Rightarrow \mu\unrhd\lambda$}   \label{eqn:Kostka_inv_triangular}
\end{equation}
(see~\cite[p.~105]{Mac1979}, for example).
\par
Now, for $\umu\in\Lambda_n$, we define the characters
\begin{equation}
\xi^{\umu}=\bigodot_{f\in\Phi}\xi^{f\mapsto\umu(f)}.   \label{eqn:xi_induction}
\end{equation}
We denote by $\chi^{\ulambda}_{\usigma}$ and $\xi^{\umu}_{\usigma}$ the characters $\chi^{\ulambda}$ and $\xi^{\umu}$, respectively, evaluated on the conjugacy class $C_{\usigma}$.
\par
We now express $\xi^{\umu}$ and $\chi^{\ulambda}$ in terms of each other. To do so, we define the \emph{shape} of $\ulambda\in\Lambda_n$ to be the mapping $s:\Phi\to\ZZ$ given by $s(f)=\abs{\ulambda(f)}$ for each $f\in\Phi$. We write $\ulambda\sim\umu$ if $\ulambda,\umu\in\Lambda_n$ have the same shape. Then~$\sim$ is an equivalence relation on~$\Lambda_n$. For $\ulambda,\umu\in\Lambda_n$ with $\ulambda\sim\umu$, write
\begin{align*}
K_{\ulambda\umu}&=\prod_{f\in\Phi}K_{\ulambda(f)\umu(f)},\\
H_{\umu\ulambda}&=\prod_{f\in\Phi}H_{\umu(f)\ulambda(f)}.
\end{align*}
We then find that
\begin{align}
\xi^{\umu}&=\sum_{\ulambda\sim\umu}K_{\ulambda\umu}\;\chi^{\ulambda}\quad\text{for each $\umu\in\Lambda_n$},   \label{eqn:uxi_from_uchi}\\
\chi^{\ulambda}&=\sum_{\umu\sim\ulambda}H_{\umu\ulambda}\;\xi^{\umu}\quad\text{for each $\ulambda\in\Lambda_n$}   \label{eqn:uchi_from_uxi}.
\end{align}
\par
An explicit expression for the degree $\chi^{\ulambda}(1)$ (where $1$ is the identity of $ G_n$) of~$\chi^{\ulambda}$ is given by the so-called $q$-analog of the hook-length formula.
\begin{lemma}[{\cite[Thm.~14]{Gre1956}}]
\label{lem:degrees_characters}
We have
\begin{equation}
\frac{1}{\chi^{\ulambda}(1)}\;\prod_{i=1}^n(q^i-1)=\prod_{f\in\Phi}\frac{1}{q^{\abs{f}b(\ulambda(f))}}\prod_{(i,j)\in\ulambda(f)}(q^{\abs{f}h_{i,j}(\ulambda(f))}-1),   \label{eqn:hook_length_formula}
\end{equation}
where, for each partition $\lambda=(\lambda_1,\lambda_2,\dots)$,
\[
b(\lambda)=\sum_{i\ge 1}(i-1)\lambda_i
\]
and $h_{i,j}(\lambda)$ is the \emph{hook length} of $\lambda$ at $(i,j)$, namely
\[
h_{i,j}(\lambda)=\lambda_i+\lambda'_j-i-j+1
\]
and the corresponding product over $(i,j)$ is over all boxes of the Young diagram of~$\ulambda(f)$.
\end{lemma}
\par
It can be readily verified from Lemma~\ref{lem:degrees_characters} that the linear (degree-one) irreducible characters of $ G_n$ are precisely the primary characters $\chi^{f\mapsto(n)}$, where $\abs{f}=1$. These are the only characters of $ G_n$ that we shall need explicitly. Let $\alpha$ be a generator of the multiplicative group $\FF_q^*$ of $\FF_q$, let $\omega=\exp(2\pi \sqrt{-1}/(q-1))$ be a complex root of unity, and let $\theta:\FF_q^*\to\CC$ be the linear character of $\FF_q^*$ given by $\theta(\alpha^i)=\omega^i$. The following result is essentially given in~\cite[pp.~415 and 444]{Gre1956}.
\begin{lemma}[{\cite{Gre1956}}]
\label{lem:det_char}
For all $g\in G_n$, we have
\[
\chi^{X-\alpha^i\mapsto(n)}(g)=\theta(\det(g)^i).
\]
In particular $\chi^{X-1\mapsto(n)}$ is the trivial character.
\end{lemma}
\par
In what follows we consider certain characters of $G_n$ related to the permutation character of $G_n$ on the set of $t$-tuples of linearly independent elements of $\FF_q^n$. For~$t\le n$, let $H_{n,t}$ be the stabiliser of a fixed $t$-tuple of linearly independent elements of~$\FF_q^n$. We define $\zeta^{(t,i)}$ to be the character obtained by inducing the linear character
\begin{equation}
\begin{gathered}
H_{n,t}\to\CC\\
g\mapsto \theta(\det(g)^i)
\end{gathered}
\label{eqn:def_char_det_H}
\end{equation}
to $ G_n$. Then $\zeta^{(t,0)}$ is the permutation character of $ G_n$ on the set of $t$-tuples of linearly independent elements of $\FF_q^n$. These characters are related to each other in the following way.
\begin{lemma}
\label{lem:perm_character_zeta}
For each $g\in G_n$, we have
\[
\zeta^{(t,i)}(g)=\theta(\det(g)^i)\zeta^{(t,0)}(g).
\]
\end{lemma}
\begin{proof}
Since similar matrices have the same determinant, we find from~\eqref{eqn:induction_formula} that
\begin{align*}
\zeta^{(t,i)}(g)&=\frac{1}{\abs{H_{n,t}}}\sums{x\in G_n\\xgx^{-1}\in H_{n,t}}\theta(\det(xgx^{-1})^i)\\
&=\frac{1}{\abs{H_{n,t}}}\sums{x\in G_n\\xgx^{-1}\in H_{n,t}}\theta(\det(g)^i)\\
&=\theta(\det(g)^i)\zeta^{(t,0)}(g).\qedhere
\end{align*}
\end{proof}
\par
We shall also need the following information about the decomposition of $\zeta^{(t,i)}$ into irreducible characters of $ G_n$.
\begin{lemma}
\label{lem:zeta_decomposition}
We have
\[
\zeta^{(t,i)}=\sums{\ulambda\in\Lambda_n}m_{i,\ulambda}\,\chi^{\ulambda},
\]
where $m_{i,\ulambda}\ne 0$ if and only if $\ulambda(\alpha^i)_1\ge n-t$.
\end{lemma}
\begin{proof}
We may choose $H_{n,t}$ to be
\[
H_{n,t}=\left\{\begin{bmatrix}
I & *\\
 & g
\end{bmatrix}
:g\in G_{n-t}
\right\},
\]
so that $H_{n,t}$ is a subgroup of the parabolic subgroup $P_{(t,n-t)}$ given in~\eqref{eqn:def_parabolic_subgroup}. Let $\pi_1$ and~$\pi_2$ be the projections onto the diagonal blocks of orders $t$ and~$n-t$, respectively, as given in~\eqref{eqn:projections}. Using Lemma~\ref{lem:det_char}, the character~\eqref{eqn:def_char_det_H} can be written as
\begin{equation}
(1\circ \pi_1)\,(\chi^{X-\alpha^i\mapsto (n-t)}\circ\pi_2).   \label{eqn:char_det_H}
\end{equation}
where $1$ is the trivial character of the trivial subgroup of $G_t$. By Frobenius reciprocity,~$1$ induces on $G_t$ to the character
\[
\sum_{\ukappa\in\Lambda_t}\chi^{\ukappa}(1)\,\chi^{\ukappa}.
\]
Since $P_{(t,n-t)}/H_{n,t}\cong G_t$, it is then readily verified that~\eqref{eqn:char_det_H} induces on $P_{(t,n-t)}$ to the character
\[
\sum_{\ukappa\in\Lambda_t}\chi^{\ukappa}(1)\,(\chi^{\ukappa}\circ\pi_1)\,(\chi^{X-\alpha^i\mapsto (n-t)}\circ\pi_2).
\]
Hence, by transitivity of induction, we have
\[
\zeta^{(t,i)}=\sum_{\ukappa\in\Lambda_t}\chi^{\ukappa}(1)\;(\chi^{\ukappa}\odot \chi^{X-\alpha^i\mapsto (n-t)}).
\]
It is well known~\cite[Ch.~IV,~\S~4]{Mac1979} that, for each fixed $f\in\Phi$, characters~$\chi^{f\mapsto \lambda}$ form an algebra with multiplication $\odot$ that is isomorphic to the algebra of symmetric functions and the images of the characters~$\chi^{f\mapsto \lambda}$ are the Schur functions. We then find from Pieri's rule (see~\cite[Ch.~I, (5.16)]{Mac1979}, for example) that
\[
\chi^{X-\alpha^i\mapsto\kappa}\odot\chi^{X-\alpha^i\mapsto(n-t)}=\sum_{\lambda}\chi^{X-\alpha^i\mapsto\lambda},
\]
where $\lambda$ runs through all partitions whose Young diagram is obtained from that of $\kappa$ by adding $n-t$ boxes, no two of which in the same column. Using~\eqref{eqn:irr_char_induction} the statement of the lemma is then readily verified.
\end{proof}


\section{The Hoffman bound}
\label{sec:association_schemes}

Henceforth we use the following notation. For a field $K$ and finite sets $X$ and~$Y$, we denote by $K(X,Y)$ the set of $\abs{X}\times\abs{Y}$ matrices $A$ with entries in $K$, where rows and columns are indexed by $X$ and $Y$, respectively. For $x\in X$ and $y\in Y$, the $(x,y)$-entry of $A$ is written as $A(x,y)$. If $\abs{Y}=1$, then we omit $Y$, so $K(X)$ is the set of column vectors $a$ indexed by $X$ and, for $x\in X$, the $x$-entry of $a$ is written as $a(x)$.
\par
The \emph{adjacency matrix} of a graph $\Gamma=(X,E)$ is the matrix $A\in\RR(X,X)$ given by
\[
A(x,y)=\begin{cases}
1 & \text{for $\{x,y\}\in E$}\\
0 & \text{otherwise}.
\end{cases}
\]
Then $A$ is a real symmetric matrix, which of course has an orthonormal system of~$\abs{X}$ eigenvectors forming a basis of $\RR(X)$. All eigenvalues of $A$ are real and referred to as the eigenvalues of $\Gamma$. Note that, if $\Gamma$ is $d$-regular, then $d$ is an eigenvalue of $\Gamma$ and the all-ones vector is a corresponding eigenvector.
\par
Our starting point arises from the following generalised versions of the Hoffman bound~\cite{Hae2021}, stated and proved by Ellis, Friedgut, and Pilpel~\cite[\S~2.4]{EllFriPil2011} in the following form.
\begin{proposition}
\label{pro:Hoffman}
Let $\Gamma=(X,E)$ be a graph on $n$ vertices. Suppose that $\Gamma_0,\Gamma_1,\dots,\Gamma_r$ are regular spanning subgraphs of $\Gamma$, all having $\{v_0,v_1,\dots,v_{n-1}\}$ as an orthonormal system of eigenvectors with $v_0$ being the all-ones vector. Let $P_i(k)$ be the eigenvalue of~$v_k$ in $\Gamma_i$. Let $w_0,w_1,\dots,w_r\in\RR$ and write $P(k)=\sum_{i=0}^rw_iP_i(k)$.
\begin{enumerate}[(i)]
\item If $Y\subseteq X$ is an independent set in $\Gamma$, then 
\[
\frac{\abs{Y}}{\abs{X}}\le \frac{\abs{P_{\min}}}{P(0)+\abs{P_{\min}}},
\]
where $P_{\min}=\min_{k\ne 0} P(k)$. In case of equality we have
\[
1_Y\in\ang{\{v_0\}\cup\{v_k:P(k)=P_{\min}\}}.
\]
\item If $Y,Z\subseteq X$ are such that there are no edges between $Y$ and $Z$ in $\Gamma$, then 
\[
\sqrt{\frac{\abs{Y}}{\abs{X}}\,\frac{\abs{Z}}{\abs{X}}}\le \frac{P_{\max}}{P(0)+P_{\max}},
\]
where $P_{\max}=\max_{k\ne 0} \abs{P(k)}$. In case of equality we have
\[
1_Y,1_Z\in\ang{\{v_0\}\cup\{v_k:\abs{P(k)}=P_{\max}\}}.
\]
\end{enumerate}
\end{proposition}
\par
In order to study graphs induced by $ G_n$ and their eigenvalues, we shall bring the theory of association schemes into play. We refer to~\cite{BanIto1984} and~\cite{GodMea2016} for background on association schemes. Every finite group gives rise to an association scheme (see~\cite[Section 2.7]{BanIto1984} or~\cite[Section 3.3]{GodMea2016} for details). We shall recall relevant background about this association scheme and its symmetrisation for~$ G_n$. 
\par
For each $\usigma\in\Lambda_n$, we define $B_{\usigma}=\CC( G_n,G_n)$ by
\[
B_{\usigma}(x,y)=\begin{cases}
1 & \text{for $x^{-1}y\in C_{\usigma}$}\\
0 & \text{otherwise}.
\end{cases}
\]
The vector space generated by $\{B_{\usigma}:\usigma\in\Lambda_n\}$ over the complex numbers turns out to be a commutative matrix algebra $\AA$, which contains the identity and the all-ones matrix and is closed under conjugate transposition. The collection of zero-one matrices~$B_{\usigma}$ therefore defines an association scheme. Since $\AA$ is commutative, it can be simultaneously diagonalised and therefore there exists a basis $\{F_{\ulambda}:\ulambda\in\Lambda_n\}$ of $\AA$ consisting of primitive idempotent matrices. These matrices are given by~\cite[Theorem~II.7.2]{BanIto1984}
\begin{equation}
F_{\ulambda}\,=\frac{\chi^{\ulambda}(1)}{\abs{ G_n}}\sum_{\usigma\in\Lambda_n}\chi^{\ulambda}_{\usigma}\,B_{\usigma}.   \label{eqn:def_F_lambda}
\end{equation}
Using the orthogonality of characters of the second kind, it is readily verified that
\begin{equation}
B_{\usigma}=\sum_{\ulambda\in\Lambda_n}\frac{\abs{C_{\usigma}}}{\chi^{\ulambda}(1)}\,\overline{\chi}^{\ulambda}_{\usigma}\,F_{\ulambda},   \label{eqn:B_from_F}
\end{equation}
where $\overline{\chi}^{\ulambda}$ is the character of $ G_n$ whose values at $g\in  G_n$ are the complex conjugates of~$\chi^{\ulambda}(g)$.
\par
For each $f\in\Phi$, let $f^*\in\Phi$ be its reciprocal polynomial, namely the monic polynomial whose roots (in an algebraic closure of $\FF_q$) are precisely the inverses of the roots of $f$. For each $\ulambda\in\Lambda_n$, define $\ulambda^*$ to be the element of $\Lambda_n$ given by $\ulambda^*(f)=\ulambda(f^*)$ for all $f\in\Phi$. We record the following lemma, in which we write $C^{-1}_{\usigma}=\{g^{-1}:g\in C_{\usigma}\}$ for $\usigma\in\Lambda_n$.
\begin{lemma}
\label{lem:conjugates}
Let $\usigma,\ulambda\in\Lambda_n$. Then we have
\begin{enumerate}[(i)]
\item $C_{\usigma^*}=C_{\usigma}^{-1}$,\\[-1.2ex]
\item $\chi^{\ulambda^*}=\overline{\chi}^{\ulambda}$,\\[-1ex]
\item $\chi^{\ulambda^*}_{\usigma}=\chi^{\ulambda}_{\usigma^*}$.
\end{enumerate}
\end{lemma}
\begin{proof}
Statement (i) is a basic fact in linear algebra, (ii) is essentially~\cite[(7.32)]{Jam1986}, and~(iii) can be deduced from (i) and (ii).
\end{proof}
\par
Let $\Omega_n$ be the subset of $\Lambda_n$ that contains all $\ulambda\in\Lambda_n$ satisfying $\ulambda=\ulambda^*$ and precisely one of $\ulambda$ or $\ulambda^*$ for all $\ulambda\in\Lambda_n$ satisfying $\ulambda\ne\ulambda^*$. For $\ulambda\in\Omega_n$, we define the character
\[
\psi^{\ulambda}=\begin{cases}
\chi^{\ulambda} & \text{for $\ulambda=\ulambda^*$}\\
\chi^{\ulambda}+\chi^{\ulambda^*} & \text{otherwise},
\end{cases}
\]
and, for $\usigma\in\Omega_n$, we define $D_{\usigma}=C_{\usigma}\cup C_{\usigma^*}$. Lemma~\ref{lem:conjugates} implies that $\psi^{\ulambda}$ is constant on~$D_{\usigma}$.  We write
\begin{equation}
\psi^{\ulambda}_{\usigma}=\psi^{\ulambda}(g),\quad\text{where $g$ is an arbitrary element of $D_{\usigma}$}.   \label{eqn:psi_constant_on_D}
\end{equation}
For $\usigma,\ulambda\in\Omega_n$, write 
\begin{equation}
A_{\usigma}=\begin{cases}
B_{\usigma} & \text{for $\usigma=\usigma^*$}\\
B_{\usigma}+B_{\usigma^*} & \text{otherwise}
\end{cases}
\quad\text{and}\quad
E_{\ulambda}=\begin{cases}
F_{\ulambda} & \text{for $\ulambda=\ulambda^*$}\\
F_{\ulambda}+F_{\ulambda^*} & \text{otherwise}.
\end{cases}
\label{eqn:def_A_and_E}
\end{equation}
Note that $A_{\usigma}$ is symmetric, so all of its eigenvalues are real, and that  $E_{\ulambda}$ has only real entries. Let $V_{\ulambda}$ be the column span over the reals of $E_{\ulambda}$ and, for $\usigma,\ulambda\in\Omega_n$, write
\begin{equation}
P(\ulambda,\usigma)=\frac{\abs{D_{\usigma}}}{\psi^{\ulambda}(1)}\;\psi^{\ulambda}_{\usigma}.      \label{eqn:ev_graph}
\end{equation}
The following lemma, containing essentially standard results, will be crucial in the following.
\begin{lemma}
\label{lem:decomposition_RX}
We have the following orthogonal direct sum decomposition
\[
\RR(G_n)=\bigoplus_{\ulambda\in\Omega_n}V_{\ulambda}.
\]
Moreover, for all $\usigma,\ulambda\in\Omega_n$, every element of $V_{\ulambda}$ is an eigenvector of $A_{\usigma}$ and the corresponding eigenvalue is $P(\ulambda,\usigma)$.
\end{lemma}
\begin{proof}
Since $F_{\ulambda}$ is a primitive idempotent in $\CC(G_n,G_n)$ for each $\ulambda\in\Lambda_n$, it is readily verified that $E_{\ulambda}$ is a primitive idempotent in $\RR(G_n,G_n)$ for each $\ulambda\in\Omega_n$. Therefore the $E_{\ulambda}$ are  pairwise orthogonal, namely we have $E_{\ulambda}E_{\umu}=\delta_{\ulambda\umu}E_{\ulambda}$ for all $\ulambda,\umu\in\Omega_n$. Since $E_{\ulambda}$ is idempotent, $\rk(E_{\ulambda})$ is just the trace of~$E_{\ulambda}$. It follows from~\eqref{eqn:def_F_lambda} that the trace of $F_{\ulambda}$ equals $\chi^{\ulambda}(1)^2$. Hence we have
\[
\sum_{\ulambda\in\Omega_n}\dim V_{\ulambda}=\sum_{\ulambda\in\Omega_n}\rk(E_{\ulambda})=\sum_{\ulambda\in\Lambda_n}\chi^{\ulambda}(1)^2=\abs{ G_n}
\]
by standard properties of the degrees of irreducible characters. This proves the first statement.  We have $\chi^{\ulambda}(1)=\chi^{\ulambda^*}(1)$ by Lemma~\ref{lem:degrees_characters}, from which together with~\eqref{eqn:B_from_F} and Lemma~\ref{lem:conjugates} it is readily verified that 
\[
A_{\usigma}=\sum_{\ulambda\in\Omega_n}P(\ulambda,\usigma)\,E_{\ulambda}.
\]
Since the $E_{\ulambda}$ are  pairwise orthogonal, we obtain the second statement.
\end{proof}
\par
In fact the proof of Lemma~\ref{lem:decomposition_RX} shows that $\{A_{\usigma}:\usigma\in\Omega_n\}$ is a symmetric association scheme with primitive idempotents given by $\{E_{\ulambda}:\ulambda\in\Omega_n\}$. However we will not exploit this further.
\par
Note that $A_{\usigma}$ is the adjacency matrix of a $\abs{D_{\usigma}}$-regular graph for each $\usigma\in\Omega_n$, except for $\usigma$ given by $\usigma(1)=(1^n)$, and that $P(\ulambda,\usigma)=\abs{D_{\usigma}}$ if $\ulambda\in\Omega_n$ is given by $X-1\mapsto (n)$. 
\par
The strategy to prove Theorems~\ref{thm:main_intersection} and~\ref{thm:main_cross_intersection} is as follows (Theorems~\ref{thm:main_intersection_space} and~\ref{thm:main_cross_intersection_space} will be proved using slight modifications). We call an element $x\in G$ a \emph{$t$-derangement} if there is no $t$-tuple of linearly independent elements of $\FF_q^n$ that is fixed by $x$. Equivalently $x\in  G_n$ is a $t$-derangement if $\rk(x-I)>n-t$. It is readily verified that either all elements of $D_{\usigma}$ are $t$-derangements or none of them. We wish to identify an appropriate subset $\Sigma$ of $\Omega_n$ such that $D_{\usigma}$ consists of $t$-derangements for all $\usigma\in\Sigma$ and then apply Proposition~\ref{pro:Hoffman} to the graph $\Gamma$ with adjacency matrix $\sum_{\usigma\in\Sigma}A_{\usigma}$ and $\abs{D_{\usigma}}$-regular spanning subgraphs $\Gamma_{\usigma}$ having adjacency matrix $A_{\usigma}$ for $\usigma\in\Sigma$. In view of Lemma~\ref{lem:decomposition_RX}, we wish to construct some $w\in\RR(\Sigma)$ such that both the minimum value and the negative of the second-largest absolute value over all $\ulambda\in\Omega_n$ of 
\begin{equation}
\sum_{\usigma\in\Sigma}w(\usigma)P(\ulambda,\usigma)   \label{eqn:LC_eigenvalues}
\end{equation}
equals
\begin{equation}
\eta=-\frac{1}{(q^n-1)(q^n-q)\cdots(q^n-q^{t-1})-1}   \label{eqn:def_min_ev}
\end{equation}
and such that $w$ is normalised in the sense that~\eqref{eqn:LC_eigenvalues} equals $1$ if $\psi^{\ulambda}$ is the trivial character (or equivalently $\ulambda\in\Omega_n$ is given by $X-1\mapsto (n)$). This will ensure that Proposition~\ref{pro:Hoffman} will give the bounds of Theorems~\ref{thm:main_intersection} and~\ref{thm:main_cross_intersection}.


\section{An invertible matrix}
\label{sec:matrix}

This section contains some key preparations for our main proofs. We first identify relevant conjugacy classes of $G_n$ whose elements are either $t$-derangements or do not fix a $t$-space. We then use these conjugacy classes to identify a matrix related to the character table of $G_n$. A key step is to show that this matrix is invertible.
\par
We call an element of $G_n$ \emph{regular elliptic} if its characteristic polynomial is irreducible. The following lemma shows that regular elliptic elements in $ G_n$ play the role of an $n$-cycle in the symmetric group $S_n$. 
\begin{lemma}[{\cite[Prop.~4.4]{LewReiSta2014}}]
\label{lem:Cf_fixing_subspace}
Each regular elliptic element of $ G_n$ fixes no proper nontrivial subspace of $\FF_q^n$. 
\end{lemma}
\par
Note that, for each $f\in\Phi$ of degree $d$, its companion matrix satisfies $\det(C_f)=(-1)^df(0)$. It is well known~\cite{HanMul1992} that, for each $a\in\FF_q^*$, there exists an irreducible polynomial $f\in\FF_q[x]$ of degree $d$ such that $f(0)=a$. Hence we can always find a polynomial in $\Phi$ with prescribed degree and prescribed nonzero determinant of its companion matrix. Also note that, for each $f\in\Phi$, we have $f(0)f^*(0)=1$ and therefore
\[
\det(C_f)\det(C_{f^*})=1.
\]
\par
We now continue to use $\alpha$ to denote a fixed generator of $\FF_q^*$. For all integers $\ell,j$ satisfying $0\le \ell<n$ and $0\le j\le q-2$, we fix an irreducible polynomial~$h_{\ell,j}\in\Phi$ of degree $n-\ell$ such that its companion matrix has determinant $\alpha^j$ and such that $h_{\ell,j}^*=h_{\ell,-j}$. We define
\[
\Sigma_{\ell,j}=\{\text{$\usigma\in\Lambda_n:\usigma(h_{\ell,j})=(1)$}\}.
\]
and 
\[
\Sigma_\ell=\bigcup_{j=0}^{q-2}\Sigma_{\ell,j}\quad\text{and}\quad \Sigma_{\le t}=\bigcup_{\ell=0}^t\Sigma_\ell.
\]
Note that, for each $\usigma\in\Sigma_{\le t-1}$, the conjugacy class $C_{\usigma}$ consists of elements that do not fix a $t$-space of $\FF_q^n$. In addition, for each $\usigma\in\Sigma_t$ with the $q-1$ exceptions $\usigma\in\Sigma_t$ satisfying $\usigma(X-1)=(1^t)$, the conjugacy class $C_{\usigma}$ consists of elements that do not fix a $t$-space pointwise. Next we define
\[
\Pi_{k,i}=\{\ulambda\in\Lambda_n:\ulambda(\alpha^i)_1=n-k\}.
\]
and
\[
\Pi_k=\bigcup_{i=0}^{q-2}\Pi_{k,i}\quad\text{and}\quad \Pi_{\le t}=\bigcup_{k=0}^t\Pi_k.
\]
Note that, for $k<n/2$, we have $\abs{\Pi_{k,i}}=\abs{\Sigma_{k,i}}$ and $\abs{\Omega_n\cap\Pi_{k,i}}=\abs{\Omega_n\cap\Sigma_{k,i}}$.
\par
We define $Q\in\RR(\Omega_n,\Omega_n)$ by
\[
Q(\ulambda,\usigma)=\psi^{\ulambda}_{\usigma}\quad\text{for each $\ulambda,\usigma\in\Omega_n$}
\]
and let $Q_t$ be the restriction of $Q$ to $\RR(\Omega_n\cap\Pi_{\le t},\Omega_n\cap\Sigma_{\le t})$. We emphasise that $Q_t$ is a square matrix. A key step in our proof is the following proposition. 
\begin{proposition}
\label{pro:matrix_full_rank}
For $n>2t$, the matrix $Q_t$ has full rank and is independent of $n$. 
\end{proposition}
\par
In the remainder of this section we essentially only prove Proposition~\ref{pro:matrix_full_rank}. The reader who is interested in maintaining the flow of the proof of our main results may wish to skip to the next section at first reading.
\par
We define $R\in\CC(\Lambda_n,\Lambda_n)$ by
\[
R(\ulambda,\usigma)=\chi^{\ulambda}_{\usigma}\quad\text{for each $\ulambda,\usigma\in\Lambda_n$}
\]
and let $R_t$ be the restriction of $R$ to $\CC(\Pi_{\le t},\Sigma_{\le t})$. We shall prove a counterpart of Proposition~\ref{pro:matrix_full_rank} for the matrix $R_t$.
\begin{proposition}
\label{pro:irr_matrix_full_rank}
For $n>2t$, the matrix $R_t$ has full rank and is independent of $n$. 
\end{proposition}
\par
Note that $Q_t$ is obtained from $R_t$ by first applying elementary row operations, then deleting some rows, and then (in view of~\eqref{eqn:psi_constant_on_D}) deleting duplicate columns. Hence Proposition~\ref{pro:matrix_full_rank} follows from Proposition~\ref{pro:irr_matrix_full_rank}.
\par
We now prove Proposition~\ref{pro:irr_matrix_full_rank}. We let $S\in\CC(\Lambda_n,\Lambda_n)$ be the matrix defined by
\begin{equation}
S(\umu,\usigma)=\xi^{\umu}_{\usigma}\quad\text{for each $\umu,\usigma\in\Lambda_n$}   \label{eqn:def_S}
\end{equation}
and let $S_t$ be the restriction of $S$ to $\CC(\Pi_{\le t},\Sigma_{\le t})$. Now recall the equivalence relation~$\sim$ on $\Lambda_n$ and the numbers $K_{\ulambda\umu}$ from Section~\ref{sec:character_theory}. Define $T\in\CC(\Lambda_n,\Lambda_n)$ by
\[
T(\umu,\ulambda)=\begin{cases}
K_{\ulambda\umu} & \text{for $\ulambda\sim\umu$}\\
0 & \text{otherwise}
\end{cases}
\]
and let~$T_t$ be the restriction of $T$ to $\CC(\Pi_{\le t},\Pi_{\le t})$. We first prove the following.
\begin{lemma}~
\label{lem:T_full_rank}
\begin{enumerate}[(i)]
\item We have $S=TR$ and $T$ has full rank.
\item For $n>2t$, we have $S_t=T_tR_t$ and $T_t$ has full rank and is independent of $n$. 
\end{enumerate}
\end{lemma}
\begin{proof}
From~\eqref{eqn:uxi_from_uchi} we have $S=TR$ and $T$ is block diagonal, where the blocks are induced by the equivalence classes under $\sim$. Each diagonal block corresponds to one equivalence class. If $s:\Phi\to\ZZ$ is the shape of such an equivalence class, then the corresponding block can be written as a Kronecker product
\[
\bigotimes_{f\in\Phi}K^{(s(f))},
\]
where $K^{(m)}\in\CC(\Par_m,\Par_m)$ is a Kostka matrix given by $K^{(m)}(\mu,\lambda)=K_{\lambda\mu}$ with the convention $K^{(0)}=(1)$ and $\Par_m$ is the set of partitions of $m$. By~\eqref{eqn:Kostka_triangular} the Kostka matrices are invertible. Hence $T$ is a block-diagonal matrix whose blocks are Kronecker products of matrices of full rank and so $T$ itself has full rank. This proves~(i).
\par
From~\eqref{eqn:Kostka_triangular} we find that $S_t=T_tR_t$.  Note that $T_t$ is still block diagonal with one diagonal block for each equivalence class of $\Lambda_n$ under $\sim$ whose shape $s:\Phi\to\ZZ$ satisfies $s(\alpha^i)\ge n-t$ for some $i$. The corresponding block can be written as
\begin{equation}
\tilde K^{(s(\alpha^i))}\otimes\bigotimes_{f\in\Phi\setminus\{\alpha^i\}}K^{(s(f))},   \label{eqn:Kronecker_T_t}
\end{equation}
where $\tilde K^{(s(\alpha^i))}$ is the matrix $K^{(s(\alpha^i))}$ restricted to partitions $\lambda$ of $s(\alpha^i)$ satisfying
\[
\lambda\unrhd (n-t,1^{s(\alpha^i)-(n-t)}).
\]
From~\eqref{eqn:Kostka_triangular} we find that, after a suitable ordering of rows and columns, all matrices occuring in the Kronecker product~\eqref{eqn:Kronecker_T_t} are upper-triangular with ones on the diagonal. Again~$T_t$ is a block-diagonal matrix whose blocks are Kronecker products of matrices of full rank and so $T_t$ itself has full rank.
\par
From the proof of~\cite[Thm.~20]{EllFriPil2011} we know that $\tilde K^{(s(\alpha^i))}$ is independent of $n$. Moreover all other matrices occuring in the Kronecker product~\eqref{eqn:Kronecker_T_t} are also independent of~$n$. Hence $T_t$ itself is also independent of $n$. This proves (ii).
\end{proof}
\par
Next we shall show that the matrix $S_t$ has full rank. Recall that, for a composition~$\lambda$, we denote by $P_{\lambda}$ the parabolic subgroup of $G_{\abs{\lambda}}$ given in~\eqref{eqn:def_parabolic_subgroup}. 
We start with the following lemma.
\begin{lemma}
\label{lem:phi_psi_ind}
Let $m$ and $n$ be positive integers satisfying $m<n$ and let $\phi$ and $\psi$ be class functions of $G_m$ and $G_n$, respectively. Let $\pi_1:P_{(m,n)}\to G_m$ and $\pi_2:P_{(m,n)}\to G_n$ be the natural projections onto the corresponding diagonal blocks. Let $g\in P_{(m,n)}$ be such that $\pi_2(g)$ is regular elliptic. Then we have
\[
(\phi\odot \psi)(g)=\phi(\pi_1(g))\,\psi(\pi_2(g)).
\]
\end{lemma}
\begin{proof}
From~\eqref{eqn:def_parabolic_induction} we have
\begin{equation}
(\phi\odot \psi)(g)=\frac{1}{\abs{P_{(m,n)}}}\sums{x\in G_{m+n}\\xgx^{-1}\in P_{(m,n)}}\phi(\pi_1(xgx^{-1}))\psi(\pi_2(xgx^{-1})).   \label{eqn:induction_phi_psi}
\end{equation}
Since $\pi_2(g)$ is regular elliptic and $m<n$, we find from Lemma~\ref{lem:Cf_fixing_subspace} that $g$ stabilises a unique $m$-dimensional subspace $U$ of $\FF_q^{m+n}$. Hence the number of $x\in G_{m+n}$ such that $xgx^{-1}\in P_{(m,n)}$ is the number of ordered bases $\{u_1,\dots,u_m,w_1,\dots,w_n\}$ of $\FF_q^{m+n}$ such that $\{u_1,\dots,u_m\}$ spans $U$. This number equals $\abs{P_{(m,n)}}$. Since $xgx^{-1}\in P_{(m,n)}$ for each $x\in P_{(m,n)}$, we conclude that 
\[
\{x\in G_{m+n}:xgx^{-1}\in P_{(m,n)}\}=P_{(m,n)}.
\]
Since $\pi_i(xgx^{-1})$ is conjugate to $\pi_i(g)$ for each $i\in\{1,2\}$ and each $x\in P_{(m,n)}$, the statement of the lemma follows from~\eqref{eqn:induction_phi_psi}.
\end{proof}
We use Lemma~\ref{lem:phi_psi_ind} to prove the following lemma on the structure of the matrix $S$.
\begin{lemma}
\label{lem:xi_reduction}
Let $k,\ell$ be integers satisfying $0\le k,\ell<n/2$ and let $\umu\in\Pi_{k,i}$ and $\usigma\in\Sigma_{\ell,j}$. If $k>\ell$, then we have $\xi^{\umu}_{\usigma}=0$. For $k\le \ell$, let $\nu$ be the partition obtained from $\umu(X-\alpha^i)$ by replacing the part $n-k$ by $\ell-k$ and define $\unu,\utau\in\Lambda_\ell$ by 
\[
\unu(f)=\begin{cases}
\nu & \text{for $f=X-\alpha^i$}\\
\umu(f) & \text{otherwise}
\end{cases}
\quad\text{and}\quad
\utau(f)=\begin{cases}
\varnothing & \text{for $f=h_{\ell,j}$}\\
\usigma(f) & \text{otherwise}.
\end{cases}
\]
If $k\le\ell$, then we have $\xi^{\umu}_{\usigma}=\xi^{\unu}_{\utau}\,\omega^{ij}$.
\end{lemma}
\begin{proof}
Let $g\in C_{\usigma}$. Define $\ukappa\in\Lambda_k$ by
\[
\ukappa(f)=\begin{cases}
(\umu(\alpha^i)_2,\umu(\alpha^i)_3,\dots) & \text{for $f=X-\alpha^i$}\\
\umu(f) & \text{otherwise},
\end{cases}
\]
so that by~\eqref{eqn:xi_induction_from_parts} and~\eqref{eqn:xi_induction}
\begin{equation}
\xi^{\umu}=\xi^{\ukappa}\odot\xi^{X-\alpha^i\mapsto (n-k)}.   \label{eqn:xi_mu}\\
\end{equation}
For $\xi^{\umu}(g)$ to be nonzero, $g$ must be conjugate to an element of $P_{(k,n-k)}$. Each such element fixes a $k$-dimensional subspace of $\FF_q^n$. If $k>\ell$, then by Lemma~\ref{lem:Cf_fixing_subspace},~$g$ fixes no $k$-dimensional subspace of $\FF_q^n$ and hence $\xi^{\umu}(g)=0$.
\par
Henceforth assume that $k\le\ell$. We shall frequently use $\xi^{f\mapsto (m)}=\chi^{f\mapsto (m)}$, which follows from~\eqref{eqn:xi_chi} and~\eqref{eqn:Kostka_triangular}. Since $k\le\ell$ we have
\begin{equation}
\xi^{\unu}=\xi^{\ukappa}\odot\xi^{X-\alpha^i\mapsto (\ell-k)}.   \label{eqn:xi_nu}
\end{equation}
Write 
\[
E=\bigcups{\urho\in\Lambda_{n-k}\\\urho(h_{\ell,j})=(1)} C_{\urho}.
\]
We claim that
\begin{equation}
\xi^{X-\alpha^i\mapsto (n-k)}(e)=(\xi^{X-\alpha^i\mapsto (\ell-k)}\odot \xi^{X-\alpha^i\mapsto (n-\ell)})(e)\quad\text{for each $e\in E$}.   \label{eqn:xi_nk_ind}
\end{equation}
Indeed, each $e\in E$ is conjugate to an element of $P_{(\ell-k,n-\ell)}$ with blocks $e_1\in G_{\ell-k}$ and $e_2\in G_{n-\ell}$ on the main diagonal, where $e_2$ is regular elliptic. Hence we find from Lemma~\ref{lem:det_char} that, for each $e\in E$, the left hand side of~\eqref{eqn:xi_nk_ind} equals
\begin{align*}
\theta(\det(e)^i)&=\theta(\det(e_1)^i)\cdot \theta(\det(e_2)^i)\\
&=\xi^{X-\alpha^i\mapsto (\ell-k)}(e_1)\cdot \xi^{X-\alpha^i\mapsto (n-\ell)}(e_2),
\end{align*}
which by Lemma~\ref{lem:phi_psi_ind} equals the right hand side of~\eqref{eqn:xi_nk_ind}. From~\eqref{eqn:xi_mu} we have
\[
\xi^{\umu}(g)=\frac{1}{\abs{P_{(k,n-k)}}}\sums{x\in G_n\\xgx^{-1}\in P_{(k,n-k)}}\xi^{\ukappa}(\pi_1(xgx^{-1}))\xi^{X-\alpha^i\mapsto (n-k)}(\pi_2(xgx^{-1})),
\]
where $\pi_1:P_{(k,n-k)}\to G_{k}$ and $\pi_2:P_{(k,n-k)}\to G_{n-k}$ are the natural projections onto the diagonal blocks. Since $k,\ell<n/2$, Lemma~\ref{lem:Cf_fixing_subspace} implies that each $\pi_2(xgx^{-1})$ occuring in the summation is forced to lie inside $E$. Hence by subsequent applications of~\eqref{eqn:xi_mu},~\eqref{eqn:xi_nk_ind}, and~\eqref{eqn:xi_nu} we then find that
\begin{align*}
\xi^{\umu}(g)&=(\xi^{\ukappa}\odot \xi^{X-\alpha^i\mapsto (n-k)})(g)\\
&=(\xi^{\ukappa}\odot\xi^{X-\alpha^i\mapsto (\ell-k)}\odot \xi^{X-\alpha^i\mapsto (n-\ell)})(g)\\
&=(\xi^{\unu}\odot \xi^{X-\alpha^i\mapsto (n-\ell)})(g).
\end{align*}
Without loss of generality, we may assume that $g\in P_{(\ell,n-\ell)}$ and that the diagonal blocks of $g$ are $g_1$ and $g_2$, where $g_1\in C_{\utau}$ and $g_2$ is the companion matrix of $h_{\ell,j}$. Since~$g_2$ is regular elliptic, we may apply Lemma~\ref{lem:phi_psi_ind} once more to obtain
\[
\xi^{\umu}(g)=\xi^{\unu}(g_1)\,\xi^{X-\alpha^i\mapsto (n-\ell)}(g_2).
\]
Since $g_1\in C_{\utau}$, we have $\xi^{\unu}(g_1)=\xi^{\unu}_{\utau}$, and since $g_2$ is the companion matrix of $h_{\ell,j}$, we find from Lemma~\ref{lem:det_char} that
\[
\xi^{X-\alpha^i\mapsto (n-\ell)}(g_2)=\theta(\det(g_2)^i)=\omega^{ij}.
\]
Hence we obtain $\xi^{\umu}(g)=\xi^{\unu}_{\utau}\,\omega^{ij}$, as required.
\end{proof}
\par
We can now prove the required property of the matrix $S_t$.
\begin{lemma}
\label{lem:R_full_rank}
For $n>2t$, the matrix $S_t$ has full rank and is independent of $n$.
\end{lemma}
\begin{proof}
To indicate dependence on $n$, write $S^{(n)}$ for the matrix $S$ given in~\eqref{eqn:def_S} and~$S^{(n)}_t$ for the corresponding restricted matrix $S_t$. Let $n>2t$. From Lemma~\ref{lem:xi_reduction} we find that all entries in $S^{(n)}_t$ are independent of~$n$, which proves the second statement of the lemma.
\par
To show that $S^{(n)}_t$ is invertible, we view $S^{(n)}_t$ as a block matrix, where the blocks are indexed by $\Pi_k$ and $\Sigma_\ell$ for $k,\ell\in\{0,1,\dots,t\}$. Let $B_{k,\ell}$ be the block corresponding to $\Pi_k$ and~$\Sigma_\ell$. Lemma~\ref{lem:xi_reduction} implies that $B_{k,\ell}$ is zero for $k>\ell$ and, for $0\le k\le t$, the block $B_{kk}$ is the Kronecker product of $S^{(k)}$ and the Vandermonde matrix $(\omega^{ij})_{0\le i,j\le q-2}$. Since the character table of irreducible characters of every finite group is invertible, Lemma~\ref{lem:T_full_rank} implies that $S^{(k)}$ is invertible and so $B_{kk}$ is invertible. Hence $S^{(n)}_t$ is block upper-triangular and all diagonal blocks are invertible. Therefore $S^{(n)}_t$ itself is invertible.
\end{proof}
\par
Finally, by combining Lemmas~\ref{lem:T_full_rank} and~\ref{lem:R_full_rank}, we obtain a proof of Proposition~\ref{pro:irr_matrix_full_rank}.


\section{Proof of Theorems~\ref{thm:main_intersection} and~\ref{thm:main_cross_intersection}}
\label{sec:proofs_points}

Now recall the definition~\eqref{eqn:ev_graph} of the eigenvalues $P(\ulambda,\usigma)$ and the definition~\eqref{eqn:def_min_ev} of the prescribed extremal eigenvalue $\eta$. As a first step in constructing the required weight function $w$ occuring in~\eqref{eqn:LC_eigenvalues}, we prove the following result.
\begin{proposition}
\label{pro:equation_system}
Let $n$ and $t$ be positive integers satisfying $n>2t$. Then there exists $w\in\RR(\Omega_n\cap\Sigma_{\le t})$ such that $w(\usigma)=0$ for $\usigma(1)=(1^t)$ and
\begin{equation}
\sum_{\usigma\in\Omega_n\cap\Sigma_{\le t}}w(\usigma)P(\ulambda,\usigma)
=
\begin{cases}
1     & \text{for $\ulambda\in\Omega_n\cap\Pi_{0,0}$}\\
\eta & \text{for $\ulambda\in\Omega_n\cap\Pi_{k,0}$ and $1\le k\le t$}\\
0     & \text{for $\ulambda\in\Omega_n\cap\Pi_{k,i}$ and $0\le k\le t$ and $1\le i\le q-2$}
\end{cases}   \label{eqn:eqn_system}
\end{equation}
and
\begin{equation}
\abs{w(\usigma)}\le \frac{\gamma_t}{\abs{D_{\usigma}}}\quad\text{for all $\usigma\in\Omega_n\cap\Sigma_{\le t}$}   \label{eqn:bound_w}
\end{equation}
for some constant $\gamma_t$ depending only on $t$.
\end{proposition}
\begin{proof}
From Proposition~\ref{pro:matrix_full_rank} we know that $Q_t$ has full rank. In view of~\eqref{eqn:ev_graph} there exists a unique $w\in\RR(\Omega_n\cap\Sigma_{\le t})$ satisfying~\eqref{eqn:eqn_system}.
\par
We now show that $w(\usigma)=0$ for the $\lfloor q/2\rfloor+1$ elements $\usigma\in\Omega_n\cap\Sigma_{\le t}$ satisfying $\usigma(1)=(1^t)$. Without loss of generality we may assume that $\Omega_n$ contains $X-\alpha^i$ and $h_{t,j}$ for all $i,j=0,1,\dots,\lfloor q/2\rfloor$. Accordingly we define $\usigma_j\in\Sigma_{t,j}$ by $\usigma_j(1)=(1^t)$ for $j=0,1,\dots,\lfloor q/2\rfloor$. Recall the definition of the character $\zeta^{(t,i)}$ from Section~\ref{sec:character_theory} and write $\zeta^{(t,i)}_{\usigma}$ for this character evaluated on the conjugacy class~$C_{\usigma}$. We evaluate the sum
\begin{equation}
S_i=\sum_{\usigma\in\Omega_n\cap\Sigma_{\le t}}w(\usigma)\abs{D_{\usigma}}\,(\zeta^{(t,i)}_{\usigma}+\zeta^{(t,-i)}_{\usigma})   \label{eqn:sum_two_ways}
\end{equation}
in two ways. Since $\zeta^{(t,0)}$ is the permutation character on the set of $t$-tuples of linearly independent elements of $\FF_q^n$, we find by Lemma~\ref{lem:perm_character_zeta} that the summand in~\eqref{eqn:sum_two_ways} is nonzero only when the elements of $C_{\usigma}$ fix a $t$-tuple of linearly independent elements of $\FF_q^n$, hence only when $\usigma=\usigma_j$ for some $j$. By the definition of $\usigma_j$, each element in~$C_{\usigma_j}$ has determinant~$\alpha^j$. Hence by applying Lemma~\ref{lem:perm_character_zeta} twice we obtain
\[
\zeta^{(t,i)}_{\usigma_j}=\omega^{ij}\zeta^{(t,0)}_{\usigma_j}=\omega^{ij}\zeta^{(t,0)}_{\usigma_0}
\]
and therefore
\begin{equation}
S_i=2\zeta^{(t,0)}_{\usigma_0}\sum_{j=0}^{\lfloor q/2\rfloor}w(\usigma_j)\,\abs{D_{\usigma_j}}\,\cos\bigg(\frac{2\pi ij}{q-1}\bigg).   \label{eqn:S_i_eqn_system}
\end{equation}
On the other hand, since $\zeta^{(t,i)}+\zeta^{(t,-i)}$ is a real-valued class function, we find from Lemma~\ref{lem:conjugates} that it is a linear combination of $\psi^{\ulambda}$ for $\ulambda\in\Omega_n$. Hence by Lemma~\ref{lem:zeta_decomposition} there exists numbers~$n_{i,\ulambda}$ such that
\[
\zeta^{(t,i)}_{\usigma}+\zeta^{(t,-i)}_{\usigma}=\sums{\ulambda\in\Omega_n\\\ulambda(\alpha^i)_1\ge n-t}n_{i,\ulambda}\,\psi^{\ulambda}_{\usigma}
\]
and hence
\begin{equation}
S_i=\sums{\ulambda\in\Omega_n\\\ulambda(\alpha^i)_1\ge n-t}n_{i,\ulambda}\sum_{\usigma\in\Omega_n\cap\Sigma_{\le t}}w(\usigma)\abs{D_{\usigma}}\,\psi^{\ulambda}_{\usigma}.   \label{eqn:S_expanded}
\end{equation}
Since~\eqref{eqn:eqn_system} holds, we conclude that $S_i=0$ for each $i$ satisfying $1\le i\le \lfloor q/2\rfloor$. Since~$\zeta^{(t,0)}$ is a permutation character, it contains the trivial character with multiplicity $1$ (this can be seen by Frobenius reciprocity, for example). Hence we have $n_{0,\ulambda}=2$ for $\ulambda\in\Omega_n$ satisfying $\ulambda(1)=(n)$. We therefore find from~\eqref{eqn:S_expanded} and~\eqref{eqn:eqn_system} that
\[
S_0=2+\eta\sums{\ulambda\in\Omega_n\\n-t\le\ulambda(1)_1<n}n_{0,\ulambda}\,\psi^{\ulambda}(1)=2+2\eta(\zeta^{(t,0)}(1)-1).
\]
Since $\zeta^{(t,0)}(1)$ equals the number of $t$-tuples of linearly independent elements of $\FF_q^n$, we have
\begin{equation}
\zeta^{(t,0)}(1)=(q^n-1)(q^n-q)\cdots(q^n-q^{t-1}).   \label{eqn:degree_zeta}
\end{equation}
Therefore $S_0=0$ and so $S_i=0$ for each $i$ satisfying $0\le i\le \lfloor q/2\rfloor$. Since each element of $C_{\usigma_0}$ fixes a $t$-tuple of linearly independent elements of $\FF_q^n$, we have $\zeta^{(t,0)}_{\usigma_0}\ne 0$. Thus~\eqref{eqn:S_i_eqn_system} implies
\[
\sum_{j=0}^{\lfloor q/2\rfloor}w(\usigma_j)\,\abs{D_{\usigma_j}}\,\,\cos\bigg(\frac{2\pi ij}{q-1}\bigg)=0\quad\text{for each $i$ satisfying $0\le i\le \lfloor q/2\rfloor$}
\]
and it is readily verified, using that $(\omega^{ij})_{0\le i,j<q-1}$ is a Vandermonde matrix, that this in turn implies that $w(\usigma_j)=0$ for all $j$ satisfying $0\le j\le \lfloor q/2\rfloor$, as required.
\par
Now, for each $\ulambda\in\Omega_n$ satisfying $n-t\le\ulambda(1)_1<n$, we find from Lemma~\ref{lem:zeta_decomposition} that
\[
\abs{\eta}\, \psi^{\ulambda}(1)\le\abs{\eta}\,(\zeta^{(t,0)}(1)-1)=1,
\]
using~\eqref{eqn:degree_zeta}. Since $\psi^{\ulambda}(1)=\chi^{\ulambda}(1)=1$ for $\ulambda\in\Pi_{0,0}$, we conclude from~\eqref{eqn:eqn_system} that
\[
\Biggabs{\sum_{\usigma\in\Omega_n\cap\Sigma_{\le t}}w(\usigma)\abs{D_{\usigma}}\,\psi^{\ulambda}_{\usigma}}\le 1   \quad\text{for each $\ulambda\in\Omega_n\cap\Pi_{\le t}$}.
\]
By Lemma~\ref{pro:matrix_full_rank} all entries of $Q_t$ (which are precisely the values of $\psi^{\ulambda}_{\usigma}$ occuring in the sum) are independent of $n$ and so are uniformly bounded by some value only depending on $t$. The same also holds for the inverse of~$Q_t$, which establishes~\eqref{eqn:bound_w}.
\end{proof}
\par
In what follows we treat the remaining eigenvalues.
\begin{lemma}
\label{lem:remaining_eigenvalues}
Let $n$ and $t$ be positive integers satisfying $n>2t$ and let $w\in\RR(\Omega_n\cap\Sigma_{\le t})$ be such that
\[
\abs{w(\usigma)}\le \frac{\gamma_t}{\abs{D_{\usigma}}}\quad\text{for all $\usigma\in\Omega_n\cap\Sigma_{\le t}$}
\]
for some constant $\gamma_t$ depending only on $t$. Then 
\[
\Biggabs{\sum_{\usigma\in\Omega_n\cap\Sigma_{\le t}}w(\usigma)P(\ulambda,\usigma)}<\abs{\eta}\quad\text{for all $\ulambda\in\Omega_n\setminus\Pi_{\le t}$},
\]
provided that $n$ is sufficiently large compared to $t$.
\end{lemma}
\par
In the proof of the lemma we use the usual scalar product on class functions of $G_n$, which is given by
\begin{equation}
\ang{\chi,\psi}=\frac{1}{\abs{ G_n}}\sum_{g\in G_n}\chi(g)\overline{\psi(g)},   \label{eqn:def_sclara_product}
\end{equation}
where $\chi,\psi$ are class functions of $G_n$.
\begin{proof}[Proof of Lemma~\ref{lem:remaining_eigenvalues}]
By the definition~\eqref{eqn:ev_graph} of $P(\ulambda,\usigma)$ and~\eqref{eqn:psi_constant_on_D} we have
\begin{equation}
P(\ulambda,\usigma)=\frac{\abs{G_n}}{\psi^{\ulambda}(1)}\,\ang{\psi^{\ulambda},1_{D_{\usigma}}}.   \label{eqn:eigenvalue_scalar_product}
\end{equation}
Since $\chi^{\ulambda}$ is irreducible, we have $\ang{\psi^{\ulambda},\psi^{\ulambda}}=1$ or $2$ and therefore we obtain, by an application of the Cauchy-Schwarz inequality,
\[
\abs{\ang{\psi^{\ulambda},1_{D_{\usigma}}}}\le \sqrt{2\,\ang{1_{D_{\usigma}},1_{D_{\usigma}}}}=\sqrt\frac{2\abs{D_{\usigma}}}{\abs{ G_n}}.
\]
From~\eqref{eqn:eigenvalue_scalar_product} and our hypothesis on $w$ we then find that
\begin{align*}
\Biggabs{\sum_{\usigma\in\Omega_n\cap\Sigma_{\le t}}w(\usigma)P(\ulambda,\usigma)}&\le \sum_{\usigma\in\Omega_n\cap\Sigma_{\le t}}\abs{w(\usigma)}\,\abs{P(\ulambda,\usigma)}\\
&\le\sum_{\usigma\in\Omega_n\cap\Sigma_{\le t}}\frac{\gamma_t}{\abs{D_{\usigma}}}\,\frac{\abs{G_n}}{\psi^{\ulambda}(1)}\sqrt{\frac{2\abs{D_{\usigma}}}{\abs{ G_n}}}\\
&\le\frac{\gamma_t\,\abs{\Sigma_{\le t}}}{\psi^{\ulambda}(1)}\max_{\usigma\in\Omega_n\cap\Sigma_{\le t}}\sqrt{\frac{2\abs{G_n}}{\abs{D_{\usigma}}}}\\
&\le\frac{\gamma_t\,\abs{\Sigma_{\le t}}}{\chi^{\ulambda}(1)}\max_{\usigma\in\Sigma_{\le t}}\sqrt{\frac{2\abs{G_n}}{\abs{C_{\usigma}}}}.
\end{align*}
Note that $\abs{\Sigma_{\le t}}$ is independent of~$n$. Using Lemmas~\ref{lem:estimation_cc} and~\ref{lem:estimate_degree}, to be stated and proved in Section~\ref{sec:estimates}, we find that there is a constant $\gamma'_t$, depending only on $t$, such that
\[
\Biggabs{\sum_{\usigma\in\Omega_n\cap\Sigma_{\le t}}w(\usigma)P(\ulambda,\usigma)}\le \frac{\gamma'_t}{q^{n/2}}\,\frac{1}{q^{nt}}
\]
for all $\ulambda\in\Omega_n\setminus\Pi_{\le t}$ and all sufficiently large $n$. The right hand side is certainly strictly smaller than $1/q^{nt}$ for all sufficiently large $n$ and the proof is completed by noting that $\abs{\eta}>1/q^{nt}$.
\end{proof}
\par
Recall that $V_{\ulambda}$ is the column span of $E_{\ulambda}$. Define
\[
U_t=\sums{\ulambda\in\Omega_n\\\ulambda(1)_1\ge n-t} V_{\ulambda}.
\]
Now we obtain the following.
\begin{theorem}
Let $t$ be a positive integer. Then, for all sufficiently large $n$, the following holds.
\begin{enumerate}[(i)]
\item Every $t$-intersecting set $Y$ in $ G_n$ satisfies
\[
\abs{Y}\le \prod_{i=t}^{n-1}(q^n-q^i)
\]
and, in case of equality, we have $1_Y\in U_t$.

\item Every pair of $t$-cross-intersecting sets $Y,Z$ in $ G_n$ satisfies
\[
\sqrt{\abs{Y}\cdot\abs{Z}}\le \prod_{i=t}^{n-1}(q^n-q^i)
\]
and, in case of equality, we have $1_Y,1_Z\in U_t$.
\end{enumerate}
\end{theorem}
\begin{proof}
As explained at the end of Section~\ref{sec:association_schemes}, we apply Proposition~\ref{pro:Hoffman} to the graph with adjacency matrix 
\[
\sums{\usigma\in\Omega_n\cap\Sigma_{\le t}\\\usigma(1)\ne (1^t)}A_{\usigma}
\]
and the $\abs{D_{\usigma}}$-regular spanning subgraphs with adjacency matrix $A_{\usigma}$ for those $\usigma$ occuring in the above set union. Since none of the elements in $D_{\usigma}$ for such $\usigma$ fix a $t$-space pointwise, every $t$-intersecting set in $G_n$ is an independent set in this graph. Recall from Lemma~\ref{lem:decomposition_RX} that every element of $V_{\ulambda}$ is an eigenvector of $A_{\usigma}$ with eigenvalue~$P(\ulambda,\usigma)$. Let $w\in\RR(\Omega_n\cap\Sigma_{\le t})$ be the vector given by Proposition~\ref{pro:equation_system} and write
\[
P(\ulambda)=\sums{\usigma\in\Omega_n\cap\Sigma_{\le t}\\\usigma(1)\ne (1^t)}w(\usigma)P(\ulambda,\usigma).
\]
Proposition~\ref{pro:equation_system} and Lemma~\ref{lem:remaining_eigenvalues} imply that, for all sufficiently large $n$, we have
\[
P(\ulambda)=
\begin{cases}
1     & \text{for $\ulambda(1)_1=n$}\\
\eta & \text{for $n-t\le \ulambda(1)_1<n$}\\
\end{cases}
\]
and $\abs{P(\ulambda)}<\abs{\eta}$ for $\ulambda(1)_1<n-t$. Hence, writing $\ulambda_0$ for $X-1\mapsto (n)$, we have $P(\ulambda_0)=1$ and
\[
\eta=\min_{\ulambda\ne \ulambda_0}P(\ulambda)\quad\text{and}\quad\abs{\eta}=\max_{\ulambda\ne \ulambda_0}\,\abs{P(\ulambda)}.
\]
Then the required result follows from Proposition~\ref{pro:Hoffman} and the decomposition of $\RR(G_n)$ given in Lemma~\ref{lem:decomposition_RX}.
\end{proof}
\par
Our proof of Theorems~\ref{thm:main_intersection} and~\ref{thm:main_cross_intersection} is completed by the following result.
\begin{theorem}
\label{thm:Ut_span}
$U_t$ is spanned by the characteristic vectors of $t$-cosets.
\end{theorem}
\begin{proof}
Let $\Ac_t$ be the set of $t$-tuples of linearly independent elements of $\FF_q^n$. Define the incidence matrix $M_t\in\CC( G_n,\Ac_t\times\Ac_t)$ of elements of $ G_n$ versus $t$-cosets by
\[
M_t(x,(u,v))=\begin{cases}
1 & \text{for $xu=v$}\\
0 & \text{otherwise},
\end{cases}
\]
so that the columns of $M_t$ are precisely the characteristic vectors of the $t$-cosets. Let $\zeta^t=\zeta^{(t,0)}$ be the permutation character of the set of $t$-tuples of linearly independent elements of $\FF_q^n$ and define $C_t\in\CC( G_n, G_n)$ by 
\[
C_t(x,y)=\zeta^t(x^{-1}y).
\]
Denoting by $\1_{xu=v}$ the indicator of the event that $x\in G_n$ maps $u$ to $v$, we have
\begin{align*}
(M_tM_t^T)(x,y)&=\sum_{u,v}M_t(x,(u,v))M_t(y,(u,v))\\
&=\sum_{u,v}\1_{xu=v}\1_{yu=v}\\
&=\sum_{u}\1_{xu=yu}\\
&=\sum_{u}\1_{x^{-1}yu=u}\\
&=\zeta^t(x^{-1}y)=C_t(x,y).
\end{align*}
Hence we have $C_t=M_tM_t^T$ and so the column span of $C_t$ equals the column span of~$M_t$ or equivalently the span of the characteristic vectors of the $t$-cosets.
\par
From Lemma~\ref{lem:zeta_decomposition} we have
\[
\zeta^t=\sums{\ulambda\in\Lambda_n\\\ulambda(1)_1\ge n-t} m_{\ulambda}\,\chi^{\ulambda}
\]
for some integers $m_{\ulambda}$ satisfying $m_{\ulambda}\ne 0$ for each $\ulambda$ occuring in the summation. Since~$\zeta^t$ is real-valued, we find by Lemma~\ref{lem:conjugates} that $m_{\ulambda^*}=m_{\ulambda}$ and therefore have
\begin{equation}
\zeta^t=\sums{\ulambda\in\Omega_n\\\ulambda(1)_1\ge n-t} m_{\ulambda}\,\psi^{\ulambda}.   \label{eqn:zeta_decomp_psi}
\end{equation}
Lemma~\ref{lem:degrees_characters} implies that $\chi^{\ulambda}(1)=\chi^{\ulambda^*}(1)$. We therefore obtain from~\eqref{eqn:def_A_and_E} and~\eqref{eqn:def_F_lambda} that
\[
E_{\ulambda}(x,y)=\frac{\chi^{\ulambda}(1)}{\abs{ G_n}}\psi^{\ulambda}(x^{-1}y)
\]
and thus find from~\eqref{eqn:zeta_decomp_psi} that
\begin{equation}
C_t=\abs{ G_n}\sums{\ulambda\in\Lambda_n\\\ulambda(1)_1\ge n-t}\frac{m_{\ulambda}}{\chi^{\ulambda}(1)}\,E_{\ulambda}.   \label{eqn:C_t_from_E_lambda}
\end{equation}
Hence the column span of $C_t$ is contained in $U_t$. Conversely, let $v$ be a column of $E_{\ukappa}$ for some $\ukappa\in\Omega_n$ satisfying $\ukappa(1)_1\ge n-t$. Since $E_{\ulambda}$ is idempotent, we have $E_{\ulambda}v=v$ for $\ukappa=\ulambda$ and Lemma~\ref{lem:decomposition_RX} implies $E_{\ulambda}v=0$ for $\ukappa\ne \ulambda$. Hence from~\eqref{eqn:C_t_from_E_lambda} we find that
\[
C_tv=\abs{ G_n}\frac{m_{\ukappa}}{\chi^{\ukappa}(1)}\,v,
\]
and, since $m_{\ukappa}\ne 0$, we conclude that $v$ is in the column span of $C_t$. This completes the proof.
\end{proof}


\section{Proof of Theorems~\ref{thm:main_intersection_space} and~\ref{thm:main_cross_intersection_space}}
\label{sec:proofs_spaces}

Our proofs of Theorems~\ref{thm:main_intersection_space} and~\ref{thm:main_cross_intersection_space} follow along similar lines as those in the previous section and therefore our proofs will be less detailed.
\par
Since the parabolic subgroup $P_{(t,n-t)}$ is the stabiliser of a $t$-space of~$\FF_q^n$, the character $\xi^{X-1\mapsto (n-t,t)}$ is the permutation character of the set of $t$-spaces of $\FF_q^n$. From~\eqref{eqn:xi_chi} we obtain its decomposition 
\begin{equation}
\xi^{X-1\mapsto (n-t,t)}=\sum_{s=0}^t\chi^{X-1\mapsto (n-s,s)}.   \label{eqn:DecompositionPermutationChar}
\end{equation}
Let $\qbin{n}{k}$ denote the $q$-binomial coefficient, which counts the number of $k$-spaces of~$\FF_q^n$. Then we have
\begin{equation}
\xi^{X-1\mapsto (n-t,t)}(1)=\qbin{n}{t},   \label{eqn:deg_perm_char_spaces}
\end{equation}
and so~\eqref{eqn:DecompositionPermutationChar} implies that
\begin{equation}
\chi^{X-1\mapsto (n-s,s)}(1)=\qbin{n}{s}-\qbin{n}{s-1}.   \label{eqn:deg_chi_char_spaces}
\end{equation}
Also note that $\psi^{X-1\mapsto \lambda}=\chi^{X-1\mapsto \lambda}$ for all partitions $\lambda$. Throughout this section, we define
\begin{align*}
\varepsilon = -\frac{1}{\qbin{n}{t}-1},
\end{align*}
which will be our prescribed extremal eigenvalue.
\par
We begin with the following counterpart of Proposition~\ref{pro:equation_system}.
\begin{proposition}
\label{pro:equation_system_space}
Let $n$ and $t$ be positive integers satisfying $n>2t$. Then there exists $w\in\RR(\Omega_n\cap\Sigma_{\le t-1})$ such that
\begin{equation}
\sum_{\usigma\in\Omega_n \cap \Sigma_{\le t-1}} w(\usigma) P(\ulambda,\usigma)= 
\begin{cases}
1                 & \text{for $\ulambda(1)=(n)$},\\
\varepsilon & \text{for $\ulambda(1)=(n-s,s)$ with $1\le s\le t$},\\
0                 & \text{for $\ulambda \in \Omega_n \cap \Pi_{\le t-1}$, where}\\
                   & \text{\quad\, $\ulambda(1)\ne (n-s,s)$ with $0\le s\leq t-1$}
\end{cases}   \label{eqn:eqn_system_space}
\end{equation}
and 
\begin{equation}
\abs{w(\usigma)} \le \frac{\gamma_t}{\abs{D_{\usigma}}} \quad \text{for all $\usigma\in\Omega_n\cap\Sigma_{\le t-1}$}.   \label{eqn:EstimationOfOmega}
\end{equation}
for some constant $\gamma_t$ depending only on $t$.
\end{proposition}
\begin{proof}
From Lemma~\ref{pro:matrix_full_rank} we know that $Q_{t-1}$ has full rank. In view of~\eqref{eqn:ev_graph} there exists a unique $w\in\RR(\Omega_n\cap\Sigma_{\le t-1})$ satisfying~\eqref{eqn:eqn_system_space} except for $\ulambda$ of the form $\ulambda(1)=(n-t,t)$.
\par
Next we show that~\eqref{eqn:eqn_system_space} also holds when $\ulambda(1)=(n-t,t)$. By Lemma~\ref{lem:xi_reduction} we have $\xi^{X-1\mapsto (n-t,t)}_{\usigma}=0$ for each $\usigma\in\Sigma_{\le t-1}$. Hence we have
\begin{align}
0&=\sum_{\usigma\in\Omega_n \cap \Sigma_{\le t-1}} w(\usigma) \abs{D_{\usigma}}\xi^{X-1\mapsto (n-t,t)}_{\usigma}   \nonumber\\
&=\sum_{s=0}^t\sum_{\usigma\in\Omega_n \cap \Sigma_{\le t-1}} w(\usigma) \abs{D_{\usigma}}\chi^{X-1\mapsto (n-s,s)}_{\usigma},   \label{eqn:sum_of_rhs}
\end{align}
using~\eqref{eqn:DecompositionPermutationChar}. Since~\eqref{eqn:eqn_system_space} holds with the only exception $\ulambda(1)=(n-t,t)$, the inner sum equals $1$ for $s=0$ and $\varepsilon\,\chi^{X-1\mapsto (n-s,s)}(1)$ for each $s$ satisfying $1\le s\le t-1$. Assuming that this is true also for $s=t$ and using~\eqref{eqn:deg_chi_char_spaces}, the right hand side of~\eqref{eqn:sum_of_rhs} is indeed
\[
1+\varepsilon\,\sum_{s=1}^t\bigg(\qbin{n}{s}-\qbin{n}{s-1}\bigg)=1+\varepsilon\,\bigg(\qbin{n}{t}-1\bigg)=0.
\]
Hence~\eqref{eqn:eqn_system_space} also holds when $\ulambda(1)=(n-t,t)$.
\par
It remains to prove~\eqref{eqn:EstimationOfOmega}. For each $s$ satisfying $1\le s\le t$, we find from~\eqref{eqn:DecompositionPermutationChar} that
\[
\abs{\varepsilon}\, \chi^{X-1\mapsto(n-s,s)}(1)\le\abs{\varepsilon}\,(\xi^{X-1\mapsto (n-t,t)}(1)-1)=1,
\]
using~\eqref{eqn:deg_perm_char_spaces}. Since $\chi^{X-1\mapsto (n)}(1)=1$, we conclude from~\eqref{eqn:eqn_system_space} that
\[
\Biggabs{\sum_{\usigma\in\Omega_n\cap\Sigma_{\le t-1}}w(\usigma)\abs{D_{\usigma}}\,\psi^{\ulambda}_{\usigma}}\le 1   \quad\text{for each $\ulambda\in\Omega_n\cap\Pi_{\le t-1}$}.
\]
By Lemma~\ref{pro:matrix_full_rank} all entries of $Q_{t-1}$ are independent of $n$ and so are uniformly bounded by some value only depending on $t$. The same also holds for the inverse of~$Q_t$, which establishes~\eqref{eqn:EstimationOfOmega}.
\end{proof}
\par
The bound~\eqref{eqn:EstimationOfOmega} and Lemma~\ref{lem:remaining_eigenvalues} ensure that the right hand side of~\eqref{eqn:eqn_system_space} is small in modulus for each $\ulambda\in\Omega_n\setminus\Pi_t$. It therefore remains to deal with the case that $\ulambda\in\Omega_n\cap\Pi_t$ except for $\ulambda\in\Omega_n$ given by $\ulambda(1)=(n-t,t)$, which is the subject of the following lemma.
\begin{lemma}
\label{lem:remaining_eigenvalues_space_t}
Let $w\in\RR(\Omega_n\cap\Sigma_{\le t-1})$ be given in Proposition~\ref{pro:equation_system_space} (so that $n>2t$). Then, for all $\ulambda\in \Omega_n\cap \Pi_t$ with $\ulambda(1)\ne (n-t,t)$, we have
\begin{align*}
\biggabs{\sum_{\usigma \in \Omega_n \cap \Sigma_{\le t-1}} w(\usigma) P(\ulambda,\usigma)}<\abs{\varepsilon},
\end{align*}
provided that $n$ is sufficiently large compared to $t$.
\end{lemma}
\begin{proof}
By slight abuse of notation, we view $w$ as an element of $\RR(G_n)$ by setting $w(x)=0$ if $x\not\in\Omega_n\cap\Sigma_{\le t-1}$ and $w(x)=w(\usigma)$ if $x\in\Omega_n\cap\Sigma_{\le t-1}$ and $x\in D_{\usigma}$. Recalling the scalar product on class functions of $G_n$ from~\eqref{eqn:def_sclara_product}, the statement of the lemma is equivalent to
\begin{equation}
\frac{\abs{G_n}}{\psi^{\ulambda}(1)}\;\abs{\ang{w,\psi^{\ulambda}}}<\abs{\varepsilon}   \label{eqn:scalar_prod_reformulated}
\end{equation}
for all $\ulambda\in \Omega_n\cap \Pi_t$ with $\ulambda(1)\ne (n-t,t)$, provided that $n$ is sufficiently large compared to $t$.
\par
Pick $\ulambda\in\Omega_n\cap \Pi_t$ such that $\ulambda(1)\ne(n-t,t)$. Then $\ulambda(\alpha^i)_1=n-t$ for some $i$. First assume that $\abs{\ulambda(1)}\ne n$. Denoting by $\re x$ the real part of a complex number $x$, we find from Lemma~\ref{lem:conjugates} and~\eqref{eqn:uchi_from_uxi} that
\[
\tfrac12\abs{\ang{w,\psi^{\ulambda}}}\le \abs{\re\;\ang{w,\chi^{\ulambda}}}=\biggabs{\sum_{\umu\sim\ulambda}H_{\umu\ulambda}\,\re\;\ang{w,\xi^{\umu}}}.
\]
Lemma~\ref{lem:xi_reduction} implies that $\xi^{\umu}_{\usigma}=0$ for each $\umu\not\in\Pi_{\le t-1}$ and each $\usigma\in\Sigma_{\le t-1}$. For $\umu\in\Lambda_n$, we have
\begin{align*}
\re\;\ang{w,\xi^{\umu}}&=\sum_{\ukappa\sim\umu}K_{\ukappa\umu}\,\re\;\ang{w,\chi^{\ukappa}}.
\end{align*}
By~\eqref{eqn:Kostka_triangular}, the summation can be taken over all $\ukappa$ such that $\ukappa(\alpha^i)\unrhd\umu(\alpha^i)$. Hence if $\mu\in\Pi_{\le t-1}$, then $\ukappa\in\Pi_{\le t-1}$. By the assumed properties of $w$ given in Proposition~\ref{pro:equation_system_space}, we have $\ang{w,\psi^{\ukappa}}=0$ for each $\ukappa\in\Omega_n\cap\Pi_{\le t-1}$ satisfying $\abs{\ukappa(1)}\ne n$. Since $\abs{\ulambda(1)}\ne n$ we conclude that $\ang{w,\psi^{\ulambda}}=0$.
\par
Now assume that $\abs{\ulambda(1)}=n$ and write $\ulambda(1)=\lambda$. 
From~\eqref{eqn:chi_xi} and~\eqref{eqn:Kostka_inv_triangular} we have
\[
\ang{w,\psi^{X-1\mapsto\lambda}}=\sums{\mu\unrhd\lambda\\\mu_1>n-t}H_{\mu\lambda}\,\ang{w,\xi^{X-1\mapsto\mu}},
\]
since by Lemma~\ref{lem:xi_reduction} in the case $\mu_1=n-t$ we have $\xi^{X-1\mapsto\mu}_{\usigma}=0$ for each $\usigma\in\Sigma_{\le t-1}$.
From~\eqref{eqn:xi_chi} and~\eqref{eqn:Kostka_triangular} we then find that
\begin{align}
\ang{w,\psi^{X-1\mapsto\lambda}}&=\sums{\mu\unrhd\lambda\\\mu_1>n-t}H_{\mu\lambda}\,\sum_{\kappa\unrhd\mu}K_{\kappa\mu}\,\ang{w,\psi^{X-1\mapsto\kappa}}\nonumber   \\
&=\frac{1}{\abs{G_n}}\sums{\mu\unrhd\lambda\\\mu_1>n-t}H_{\mu\lambda}+\sums{\mu\unrhd\lambda\\\mu_1>n-t}H_{\mu\lambda}\sum_{(n)\rhd\kappa\unrhd\mu}K_{\kappa\mu}\,\ang{w,\psi^{X-1\mapsto\kappa}},   \label{eqn:scalar_prod_H_K}
\end{align}
using that $\abs{G_n}\,\ang{w,\psi^{X-1\mapsto (n)}}=1$ by the assumed properties of $w$ given in Proposition~\ref{pro:equation_system_space} and $K_{(n)\mu}=1$ for each partition $\mu$ of $n$. We first show that the first sum is zero. We have
\begin{equation}
\sums{\mu\unrhd\lambda\\\mu_1>n-t}H_{\mu\lambda}=\sums{\mu\unrhd\lambda}K_{(n)\mu}H_{\mu\lambda}-\sums{\mu\unrhd\lambda}K_{(n-t,t)\mu}H_{\mu\lambda},   \label{eqn:sum_H}
\end{equation}
using that $\lambda_1=n-t$ and that, for each partition $\mu$ of $n$, we have
\[
K_{(n-t,t)\mu}=\begin{cases}
1 & \text{for $\mu_1=n-t$}\\
0 & \text{for $\mu_1>n-t$}.
\end{cases}
\]
It is readily verified that
\begin{equation}
\sum_{\mu\unrhd\lambda} K_{\kappa\mu}H_{\mu\lambda}=\delta_{\kappa\lambda}.   \label{eqn:HK_inverse}
\end{equation}
Since $\lambda$ is neither $(n)$ nor $(n-t,t)$, we conclude that~\eqref{eqn:sum_H} equals zero. Hence~\eqref{eqn:scalar_prod_H_K} becomes
\begin{equation}
\ang{w,\psi^{X-1\mapsto\lambda}}=\sums{\mu\unrhd\lambda\\\mu_1>n-t}H_{\mu\lambda}\sum_{(n)\rhd\kappa\unrhd\mu}K_{\kappa\mu}\,\ang{w,\psi^{X-1\mapsto\kappa}}   \label{eqn:scalar_prod_H_K_2}.
\end{equation}
By the assumed properties of $w$ given in Proposition~\ref{pro:equation_system_space}, the inner summand is nonzero only when $\kappa=(n-s,s)$ for some $s$ satisfying $1\le s\le t-1$. In particular, for $\kappa$ of this form, Proposition~\ref{pro:equation_system_space} and~\eqref{eqn:deg_chi_char_spaces} give
\[
\abs{G_n}\,\abs{\ang{w,\psi^{X-1\mapsto \kappa}}}=\frac{\qbin{n}{s}-\qbin{n}{s-1}}{\qbin{n}{t}-1}\le \frac{\qbin{n}{t-1}}{\qbin{n}{t}}=\frac{q^t-1}{q^{n-t+1}-1}\le \frac{q^{2t-1}}{q^n}.
\]
By Lemma~\ref{lem:T_full_rank} the Kostka numbers $K_{\kappa\mu}$ occuring in~\eqref{eqn:scalar_prod_H_K_2} are independent of $n$ and it is readily verified from~\eqref{eqn:HK_inverse} that the numbers $H_{\mu\lambda}$ occuring in~\eqref{eqn:scalar_prod_H_K_2} are also independent of $n$. Moreover the number of summands in~\eqref{eqn:scalar_prod_H_K_2} is also independent of~$n$. From Lemma~\ref{lem:estimate_degree}, to be stated and proved in Section~\ref{sec:estimates}, we have $\psi^{X-1\mapsto\lambda}(1)\ge \delta_{t-1}\,q^{nt}$ for some constant $\delta_{t-1}$ only depending on $t$. Hence there is a constant $c_t$, depending only on $t$, such that
\[
\frac{\abs{G_n}}{\psi^{X-1\mapsto\lambda}(1)}\,\abs{\ang{w,\psi^{X-1\mapsto\lambda}}}\le \frac{c_t}{q^{n(t+1)}}.
\]
Since $\abs{\varepsilon}>1/q^{nt}$, this shows that~\eqref{eqn:scalar_prod_reformulated} holds provided that $n$ is sufficiently large compared to $t$.
\end{proof}
\par
Recall that $V_{\ulambda}$ is the column span of $E_{\ulambda}$. Define
\[
W_t=\sums{\ulambda\in\Omega_n\\\ulambda(1)\unrhd (n-t,t)} V_{\ulambda}.
\]
Now we obtain the following.
\begin{theorem}
Let $t$ be a positive integer. Then, for all sufficiently large $n$, the following holds.
\begin{enumerate}[(i)]
\item Every $t$-space-intersecting set $Y$ in $ G_n$ satisfies
\[
\abs{Y}\le \Bigg[\prod_{i=0}^{t-1}(q^t-q^i)\Bigg]\Bigg[\prod_{i=t}^{n-1}(q^n-q^i)\Bigg]
\]
and, in case of equality, we have $1_Y\in W_t$.

\item Every pair of $t$-space-cross-intersecting sets $Y,Z$ in $ G_n$ satisfies
\[
\sqrt{\abs{Y}\cdot\abs{Z}}\le \Bigg[\prod_{i=0}^{t-1}(q^t-q^i)\Bigg]\Bigg[\prod_{i=t}^{n-1}(q^n-q^i)\Bigg]
\]
and, in case of equality, we have $1_Y,1_Z\in W_t$.
\end{enumerate}
\end{theorem}
\begin{proof}
We apply Proposition~\ref{pro:Hoffman} to the graph with adjacency matrix 
\[
\sum_{\usigma\in\Omega_n\cap\Sigma_{\le t-1}}A_{\usigma}
\]
and the $\abs{D_{\usigma}}$-regular spanning subgraphs with adjacency matrix $A_{\usigma}$ for those $\usigma$ occuring in the above set union. Every $t$-space-intersecting set in $G_n$ is an independent set in this graph. Let $w\in\RR(\Omega_n\cap\Sigma_{\le t-1})$ be given by Proposition~\ref{pro:equation_system_space} and write
\[
P(\ulambda)=\sums{\usigma\in\Omega_n\cap\Sigma_{\le t-1}}w(\usigma)P(\ulambda,\usigma).
\]
Proposition~\ref{pro:equation_system_space} and Lemmas~\ref{lem:remaining_eigenvalues} and~\ref{lem:remaining_eigenvalues_space_t} imply that, for all sufficiently large $n$, we have
\[
P(\ulambda)=
\begin{cases}
1                 & \text{for $\ulambda(1)=(n)$}\\
\varepsilon & \text{for $\ulambda(1)=(n-s,s)$ with $1\le s\le t$}\\
\end{cases}
\]
and $\abs{P(\ulambda)}<\abs{\varepsilon}$ for $\ulambda(1)\ne (n-s,s)$ with some $s$ satisfying $0\le s\le t$. Hence, writing~$\ulambda_0$ for $X-1\mapsto (n)$, we have $P(\ulambda_0)=1$ and
\[
\varepsilon=\min_{\ulambda\ne \ulambda_0}P(\ulambda)\quad\text{and}\quad\abs{\varepsilon}=\max_{\ulambda\ne \ulambda_0}\,\abs{P(\ulambda)}.
\]
Then the required result follows from Proposition~\ref{pro:Hoffman} and the decomposition of $\RR(G_n)$ given in Lemma~\ref{lem:decomposition_RX}.
\end{proof}
\par
Our proof of Theorems~\ref{thm:main_intersection_space} and~\ref{thm:main_cross_intersection_space} is completed by the following result.
\begin{theorem}
$W_t$ is spanned by the characteristic vectors of cosets of stabilisers of $t$-spaces.
\end{theorem}
\begin{proof}
The proof is almost identical to that of Theorem~\ref{thm:Ut_span} with $\Ac_t$ replaced by the set of $t$-spaces and $\zeta^t$ replaced by the permutation character $\xi^{X-1\mapsto(n-t,t)}$ of $t$-spaces and the decomposition of $\zeta^t$ replaced by the decomposition given in~\eqref{eqn:DecompositionPermutationChar}. 
\end{proof}


\section{Estimates on conjugacy class sizes and character degrees}
\label{sec:estimates}

In this section we provide bounds on the size of certain conjugacy classes and degrees of certain irreducible characters of $ G_n$, which are used in the proof of Lemma~\ref{lem:remaining_eigenvalues}.
\begin{lemma} 
\label{lem:estimation_cc}
Let $n$ and $t$ be positive integers satisfying $n>2t$ and let $\usigma\in\Sigma_{\le t}$. Then we have
\[
\frac{\abs{ G_n}}{\abs{C_{\usigma}}}\le q^{t^5} q^n.
\]
\end{lemma}
\begin{proof}
From Lemma~\ref{lem:sizes_cc} we find that (with the same notation as in Lemma~\ref{lem:sizes_cc})
\begin{equation}
\frac{\abs{ G_n}}{\abs{C_{\usigma}}}\le\prod_{f\in\Phi} \prod_{i=1}^{\abs{\usigma(f)}} q^{\abs{f}\, s_i(\usigma(f)')\,m_i(\usigma(f))}.   \label{eqn:upper_bound_cc}
\end{equation}
Since $\usigma\in\Sigma_{\le t}$ and $t<n/2$, there is exactly one polynomial $h\in\Phi$ of degree at least $n-t$ in the support of $\usigma$. This polynomial must satisfy $\usigma(h)=(1)$ and the corresponding factor in~\eqref{eqn:upper_bound_cc} is at most $q^n$. There are at most $t$ other polynomials in the support of $\usigma$. Each such polynomial $f$ has degree at most $t$ and satisfies $\abs{\usigma(f)}\le t$ and hence the corresponding factor in~\eqref{eqn:upper_bound_cc} has a crude upper bound of $q^{t^4}$. As there are at most $t$ such factors, the proof is completed.
\end{proof}
\par
\begin{lemma}
\label{lem:estimate_degree}
Let $t$ be a positive integer. Then there is a constant $\delta_t$ such that, for all sufficiently large $n$ and for all $\ulambda\in\Lambda_n\setminus \Pi_{\le t}$, we have
\[
\chi^{\ulambda}(1)\ge \delta_t\,q^{n(t+1)}.
\]
\end{lemma}
\begin{proof}
Let $\ulambda\in\Lambda_n\setminus \Pi_{\le t}$. Using elementary calculus we find that
\[
1-x\ge 4^{-x}\quad\text{for $0\le x\le 1/2$}
\]
and therefore
\[
\frac{\prod_{i=1}^n(q^i-1)}{q^{\frac{1}{2}n(n+1)}}=\prod_{i=1}^n\bigg(1-\frac{1}{q^i}\bigg)\ge \prod_{i=1}^n\bigg(1-\frac{1}{2^i}\bigg)\ge \prod_{i=1}^n4^{-1/2^i}\ge\prod_{i=1}^\infty4^{-1/2^i}=\frac{1}{4}.
\]
Substitute into~\eqref{eqn:hook_length_formula} of Lemma~\ref{lem:degrees_characters} to give
\begin{equation}
\frac{1}{\chi^{\ulambda}(1)}\le 4q^{N(\ulambda)-M(\ulambda)-\frac{1}{2}n(n+1)},   \label{eqn:bound_f_n_M}
\end{equation}
where
\begin{align*}
N(\ulambda)&=\sum_{f\in\Phi}\abs{f}\sum_{(i,j)\in\ulambda(f)}h_{i,j}(\ulambda(f)),\\
M(\ulambda)&=\sum_{f\in\Phi}\abs{f}\,b(\ulambda(f))
\end{align*}
and $b$ and $h_{i,j}$ are as in  Lemma~\ref{lem:degrees_characters}. Note that for each partition $\lambda$, we have
\begin{equation}
\sum_{(i,j)\in\lambda}h_{i,j}(\lambda)\le \sum_{k=1}^{\abs{\lambda}}k=\frac{1}{2}\abs{\lambda}(\abs{\lambda}+1).   \label{eqn:bound_sum_of_hook_lengths}
\end{equation}
First assume that there exists a polynomial $h\in\Phi$ such that $\abs{h}=1$ and $\ulambda(h)'_1\ge n-t$. In this case we have
\[
M(\ulambda)\ge b(\ulambda(h))\ge \sum_{k=1}^{n-t-1}k=\frac{1}{2}(n-t)(n-t-1)
\]
and by~\eqref{eqn:bound_sum_of_hook_lengths}
\begin{align*}
N(\ulambda)&\le \frac{1}{2}\sum_{f\in\Phi}\abs{f}\abs{\ulambda(f)}(\abs{\ulambda(f)}+1)\\
&\le \frac{n+1}{2}\sum_{f\in\Phi}\abs{f}\abs{\ulambda(f)}\\
&=\frac{n(n+1)}{2}.
\end{align*}
Therefore~\eqref{eqn:bound_f_n_M} implies that
\[
\frac{1}{\chi^{\ulambda}(1)}\le 4q^{-\frac{1}{2}(n-t)(n-t-1)},
\]
so that we have $\chi^{\ulambda}(1)\ge q^{n(t+1)}$ for all sufficiently large $n$ by very crude estimates.
\par
Hence we can assume that $\ulambda(f)'_1\le n-t-1$ and $\ulambda(f)_1\le n-t-1$ for all $f\in\Phi$ satisfying $\abs{f}=1$. Note that the second assumption is implied by the hypothesis $\ulambda\not\in\Pi_{\le t}$. In what follows we use the trivial bound $M(\ulambda)\ge 0$. We distinguish two cases.
\par
In the first case we assume that $\abs{\ulambda(f)}\le n-t-1$ for all $f\in\Phi$ satisfying $\abs{f}=1$. Let $\ell$ be the maximum of $\abs{\ulambda(f)}$ over all $f\in\Phi$ satisfying $\abs{f}=1$, hence $\ell\le n-t-1$. By~\eqref{eqn:bound_sum_of_hook_lengths} we have
\begin{align*}
N(\ulambda)&\le\frac{1}{2}\sum_{f\in\Phi}\abs{f}\abs{\ulambda(f)}(\abs{\ulambda(f)}+1)\\
&=\frac{n}{2}+\frac{1}{2}\sum_{f\in\Phi}\abs{f}\abs{\ulambda(f)}^2.
\end{align*}
If $\ell\le n/2$, then we have $\abs{\ulambda(f)}\le n/2$ for all $f\in\Phi$ and so $N(\ulambda)\le n^2/4+n/2$. From~\eqref{eqn:bound_f_n_M} we then find that $\chi^{\ulambda}(1)\ge q^{n(t+1)}$ for all sufficiently large $n$, again by very crude estimates. If $\ell>n/2$, then
\begin{align*}
N(\ulambda)&\le \frac{1}{2}(n+\ell^2+(n-\ell)^2)\\
&\le \frac{1}{2}(n+(n-t-1)^2+(t+1)^2)\\
&=\frac{n^2+n}{2}-n(t+1)+(t+1)^2,
\end{align*}
where we have used that $x^2+(n-x)^2$ is increasing for $x\ge n/2$. Hence in this case we obtain $1/\chi^{\ulambda}(1)\le 4q^{-n(t+1)+(t+1)^2}$ by~\eqref{eqn:bound_f_n_M}.
\par
In the remaining case we assume that there exists $h\in\Phi$ such that $\abs{h}=1$ and $\abs{\ulambda(h)}\ge n-t$. Recall that we also assume that $\ulambda(h)_1\le n-t-1$ and $\ulambda(h)'_1\le n-t-1$. Since $N(\ulambda)$ depends only on the hook lengths of $\ulambda(f)$ for $f\in\Phi$, we may replace $\ulambda(h)$ by its conjugate $\ulambda(h)'$. Assuming that $n$ is sufficiently large, namely $n\ge (t+2)^2$, we have $\ulambda(h)_1\ge t+2$ or $\ulambda(h)'_1\ge t+2$ and we assume without loss of generality that $\ulambda(h)_1\ge t+2$. Write $\ulambda(h)_1=n-r$, so that our assumptions imply $t+1\le r\le n-t-2$. Then, writing $s=\abs{\ulambda(h)}$, there exist nonnegative integers~$c_j$ satisfying
\[
\sum_{j=1}^{n-r}h_{1j}(\ulambda(h))=\sum_{j=1}^{n-r}(j+c_j),
\quad\text{where}\quad \sum_{j=1}^{n-r}c_j=s-(n-r).
\]
Hence
\[
\sum_{j=1}^{n-r}h_{1j}(\ulambda(h))=\binom{n-r+1}{2}+(s-n+r).
\]
Application of~\eqref{eqn:bound_sum_of_hook_lengths} with $\lambda=(\ulambda(h)_2,\ulambda(h)_3,\dots)$ gives
\begin{align*}
\sum_{(i,j)\in\ulambda(h)}h_{i,j}(\ulambda(h))&\le \binom{s-n+r+1}{2}+\binom{n-r+1}{2}+(s-n+r)\\
&=\frac{s^2}{2}+\frac{3s}{2}+n^2-sn-n+r(r-(2n-s-1))\\
&\le \frac{s^2}{2}+\frac{3s}{2}+n^2-sn-n+(t+1)((t+1)-(2n-s-1)),
\end{align*}
since the term depending on $r$ is maximised for $r=t+1$ over the interval $[t+1,n-t-2]$. This last expression equals
\[
\frac{s}{2}+\frac{1}{2}s(s-2(n-t-2))+n^2-n+(t+1)((t+1)-(2n-1)).
\]
The second summand is increasing for $s\ge n-t$ and so is at most $\frac{1}{2}n(n-2(n-t-2))$. Hence we obtain
\[
\sum_{(i,j)\in\ulambda(h)}h_{i,j}(\ulambda(h))\le \frac{s}{2}+\frac{n^2}{2}-n(t+1)+(t+1)(t+2).
\]
Invoking~\eqref{eqn:bound_sum_of_hook_lengths} once more, we obtain
\[
N(\ulambda)\le \sum_{(i,j)\in\ulambda(h)}h_{ij}+\frac{1}{2}\sums{f\in\Phi\\f\ne h}\abs{f}\abs{\ulambda(f)}(\abs{\ulambda(f)}+1).
\]
We have
\[
\frac{s}{2}+\frac{1}{2}\sums{f\in\Phi\\f\ne h}\abs{f}\abs{\ulambda(f)}=\frac{1}{2}\sum_{f\in\Phi}\abs{f}\abs{\ulambda(f)}=\frac{n}{2}
\]
and
\[
\frac{1}{2}\sums{f\in\Phi\\f\ne h}\abs{f}\abs{\ulambda(f)}^2\le \frac{1}{2}\bigg(\sums{f\in\Phi\\f\ne h}\abs{f}\abs{\ulambda(f)}\bigg)^2\le \frac{t^2}{2}.
\]
Collecting all terms, we find that
\[
N(\ulambda)\le \frac{n(n+1)}{2}-n(t+1)+(t+1)(t+2)+\frac{t^2}{2}.
\]
From~\eqref{eqn:bound_f_n_M} we then obtain
\[
\frac{1}{\chi^{\ulambda}(1)}\le 4q^{-n(t+1)+(t+1)(t+2)+\frac{1}{2}t^2},
\]
which completes the proof.
\end{proof}



\begin{thebibliography}{10}

\bibitem{Aha2022}
M.~Ahanjideh, \emph{On the largest intersecting set in {${\rm GL}_2(q)$} and
  some of its subgroups}, C. R. Math. Acad. Sci. Paris \textbf{360} (2022),
  497--502.

\bibitem{AhaAha2014}
M.~Ahanjideh and N.~Ahanjideh, \emph{{Erd\H{o}s-Ko-Rado} theorem in some linear
  groups and some projective special linear group}, Studia Sci. Math. Hungar.
  \textbf{51} (2014), no.~1, 83--91.

\bibitem{AhmMea2015}
B.~Ahmadi and K.~Meagher, \emph{The {Erd\H{o}s-Ko-Rado} property for some
  permutation groups}, Australas. J. Combin. \textbf{61} (2015), 23--41.

\bibitem{BanIto1984}
E.~Bannai and T.~Ito, \emph{Algebraic combinatorics. {I}: Association schemes},
  The Benjamin/Cummings Publishing Co., Inc., Menlo Park, CA, 1984.

\bibitem{CamKu2003}
P.~J. Cameron and C.~Y. Ku, \emph{Intersecting families of permutations},
  European J. Combin. \textbf{24} (2003), no.~7, 881--890.

\bibitem{Ell2012}
D.~Ellis, \emph{Setwise intersecting families of permutations}, J. Combin.
  Theory Ser. A \textbf{119} (2012), no.~4, 825--849.

\bibitem{EllFriPil2011}
D.~Ellis, E.~Friedgut, and H.~Pilpel, \emph{Intersecting families of
  permutations}, J. Amer. Math. Soc. \textbf{24} (2011), no.~3, 649--682.

\bibitem{EllKinLif2022}
D.~Ellis, G.~Kindler, and N.~Lifshitz, \emph{Forbidden intersection problems
  for families of linear maps}, 2022, arXiv: arXiv:2208.04674 [math.CO].

\bibitem{ErdKoRad1961}
P.~Erd\H{o}s, C.~Ko, and R.~Rado, \emph{Intersection theorems for systems of
  finite sets}, Quart. J. Math. Oxford Ser. (2) \textbf{12} (1961), 313--320.

\bibitem{DezFra1977}
P.~Frankl and M.~Deza, \emph{On the maximum number of permutations with given
  maximal or minimal distance}, J. Combin. Theory Ser. A \textbf{22} (1977),
  no.~3, 352--360.

\bibitem{Gel1970}
S.~I. Gelfand, \emph{Representations of the full linear group over a finite
  field}, Mat. Sb. (N.S.) \textbf{83 (125)} (1970), 15--41.

\bibitem{GodMar2009}
Ch. Godsil and K.~Meagher, \emph{A new proof of the {Erd\H{o}s}-{K}o-{R}ado
  theorem for intersecting families of permutations}, European J. Combin.
  \textbf{30} (2009), no.~2, 404--414.

\bibitem{GodMea2016}
\bysame, \emph{{Erd\H{o}s-Ko-Rado} theorems: Algebraic approaches}, Cambridge
  Studies in Advanced Mathematics, vol. 149, Cambridge University Press,
  Cambridge, 2016.

\bibitem{Gre1956}
J.~A. Green, \emph{The characters of the finite general linear groups}, Trans.
  Amer. Math. Soc. \textbf{80} (1955), 402--447.

\bibitem{Hae2021}
W.~H. Haemers, \emph{Hoffman's ratio bound}, Linear Algebra Appl. \textbf{617}
  (2021), 215--219.

\bibitem{HanMul1992}
T.~Hansen and G.~Mullen, \emph{Primitive polynomials over finite fields}, Math.
  Comp. \textbf{59} (1992), 639--643.

\bibitem{Jam1986}
G.~James, \emph{The irreducible representations of the finite general linear
  groups}, Proc. London Math. Soc. (3) \textbf{52} (1986), no.~2, 236--268.

\bibitem{LarMal2004}
B.~Larose and C.~Malvenuto, \emph{Stable sets of maximal size in {K}neser-type
  graphs}, European J. Combin. \textbf{25} (2004), no.~5, 657--673.

\bibitem{LewReiSta2014}
J.~B. Lewis, V.~Reiner, and D.~Stanton, \emph{Reflection factorizations of
  {S}inger cycles}, J. Algebraic Combin. \textbf{40} (2014), no.~3, 663--691.

\bibitem{Mac1979}
I.~G. Macdonald, \emph{Symmetric functions and {H}all polynomials}, second ed.,
  Oxford Classic Texts in the Physical Sciences, The Clarendon Press, Oxford
  University Press, New York, 2015.

\bibitem{MeaRaz2021}
K.~Meagher and A.~S. Razafimahatratra, \emph{Some {Erd\H{o}s-Ko-Rado} results
  for linear and affine groups of degree two}, Art Discrete Appl. Math.
  \textbf{6} (2023), no.~1, Paper No. 1.05, 30 pp.

\bibitem{MeaSpi2011}
K.~Meagher and P.~Spiga, \emph{An {Erd\H{o}s}-{K}o-{R}ado theorem for the
  derangement graph of {${\rm PGL}(2,q)$} acting on the projective line}, J.
  Combin. Theory Ser. A \textbf{118} (2011), no.~2, 532--544.

\bibitem{MeaSpi2014}
\bysame, \emph{An {Erd\H{o}s}-{K}o-{R}ado theorem for the derangement graph of
  {${\rm PGL}_3(q)$} acting on the projective plane}, SIAM J. Discrete Math.
  \textbf{28} (2014), no.~2, 918--941.

\bibitem{Sag2001}
B.~E. Sagan, \emph{The symmetric group}, 2nd ed., Graduate Texts in
  Mathematics, vol. 203, Springer-Verlag, New York, 2001, Representations,
  combinatorial algorithms, and symmetric functions.

\bibitem{Spi2019}
P.~Spiga, \emph{The {Erd\H{o}s}-{K}o-{R}ado theorem for the derangement graph
  of the projective general linear group acting on the projective space}, J.
  Combin. Theory Ser. A \textbf{166} (2019), 59--90.

\bibitem{Sta2012}
R.~P. Stanley, \emph{Enumerative combinatorics. {V}olume 1}, 2nd ed., Cambridge
  Studies in Advanced Mathematics, vol.~49, Cambridge University Press,
  Cambridge, 2012.

\bibitem{Wil1984}
R.~M. Wilson, \emph{The exact bound in the {Erd\H{o}s-Ko-Rado} theorem},
  Combinatorica \textbf{4} (1984), no.~2-3, 247--257.

\end{thebibliography}

\providecommand{\bysame}{\leavevmode\hbox to3em{\hrulefill}\thinspace}
\providecommand{\MR}{\relax\ifhmode\unskip\space\fi MR }
\providecommand{\MRhref}[2]{%
  \href{http://www.ams.org/mathscinet-getitem?mr=#1}{#2}
}
\providecommand{\href}[2]{#2}

\end{document}